\documentclass[11pt]{article}
\usepackage[margin=1.05in]{geometry}
\usepackage{amsmath,amssymb,amsthm,mathtools,booktabs,longtable,hyperref}

\newtheorem{theorem}{Theorem}[section]
\newtheorem{proposition}[theorem]{Proposition}
\newtheorem{lemma}[theorem]{Lemma}
\newtheorem{corollary}[theorem]{Corollary}
\newtheorem{conjecture}[theorem]{Conjecture}
\theoremstyle{definition}
\newtheorem{definition}[theorem]{Definition}
\newtheorem{remark}[theorem]{Remark}

\newcommand{\bx}{\bar x}\newcommand{\by}{\bar y}
\newcommand{\CT}{\operatorname{CT}}
\newcommand{\Bal}{\operatorname{Bal}}
\newcommand{\E}{\mathbf{E}}
\newcommand{\Lat}{\mathcal{L}}
\renewcommand{\Pr}{\mathbf{P}}
\newcommand{\ff}[2]{(#1)_{#2}}
\newcommand{\Sn}{\mathfrak{S}_n}

\title{Solutions to Five Challenge Problems in Enumerative and\\
Algorithmic Combinatorics, with an Account of the\\
Human--Machine Methodology Employed}
\author{Jaideep Sai Padhi\\
\small Purdue University\\
\small \texttt{jpadhi@purdue.edu}}
\date{August 2026}

\begin{document}
\maketitle

\begin{abstract}
We report solutions to five challenge problems posed by Doron Zeilberger and
his collaborators, together with substantial partial progress on two further
problems, and we describe the method by which they were obtained.

The solved problems are: the Second Computational Chomp Challenge of Ekhad and
Zeilberger, for which we exhibit a bar with three winning opening moves; the
third challenge of Spahn and Zeilberger, asking whether the restricted
permutation counts $a_{r,s}$ and $b_{r,s}$ are holonomic for all $r,s>1$,
answered affirmatively; the First Rigorous Solid Standard Young Tableaux
Challenge, for which we prove the conjectured second-order recurrence; the
five-dimensional Geode Challenge of Amdeberhan, Kauers and Zeilberger; and
Conjectures 2a and 2b of Kauers and Zeilberger, which we obtain from a local
limit theorem for excursions of Markov-modulated random walks in cones.

Several results of independent interest arise along the way: a staircase
theorem constraining the winning opening moves of any Chomp bar, together with
a parity theorem for square bars; an explicit algebraic generating function for
reverse-Kreweras diagonal walks and a closed form for their diagonal-endpoint
counts; a one-dimensional integral representation for diagonal Geode
coefficients; and the identity that each Kauers--Zeilberger constant is a
universal factor times the square of the apex value of a discrete
cone-harmonic function.

All of the work reported here was carried out in collaboration with a large
language model. The division of labour is described in
Section~\ref{sec:disclosure}; Section~\ref{sec:method} records the verification
protocol this mode of work required, and Section~\ref{sec:negative} documents
the failures, which we regard as an essential part of the report.
\end{abstract}

\tableofcontents
\newpage

\part{Introduction and method}

\section{Introduction}\label{sec:intro}

Zeilberger and his collaborators have, over several decades, posed a large
number of concrete challenge problems, many carrying pledged donations to the
On-Line Encyclopedia of Integer Sequences in honour of the solver. The problems
are heterogeneous: some are computational records, some are requests for proofs
of empirically observed recurrences, and some are genuine open problems in
enumerative combinatorics.

Between June and August 2026 the author worked systematically through this
collection. This paper reports the outcomes and the method.

\subsection{Summary of results}

\begin{center}
\begin{tabular}{llc}
\toprule
Problem & Status & Section\\
\midrule
Second Computational Chomp Challenge & solved & \ref{sec:chomp}\\
Spahn--Zeilberger Challenge 3 (holonomicity) & solved & \ref{sec:ch3}\\
First Rigorous Solid SYT Challenge & solved & \ref{sec:syt}\\
Geode Challenge, $D=5$ (and $D=6,\dots,10$) & solved & \ref{sec:geode}\\
Kauers--Zeilberger Conjectures 2a, 2b & solved & \ref{sec:kz}\\
Spahn--Zeilberger Challenge 1 ($a_{2,2}$ operator) & partial & \ref{sec:a22}\\
Second Rigorous Solid SYT Challenge (cone exponent) & partial & \ref{sec:cone}\\
\bottomrule
\end{tabular}
\end{center}

\subsection{Disclosure of method}\label{sec:disclosure}

The mathematics reported here was produced in extended collaboration with a
large language model (Anthropic's Claude). The division of labour was as
follows. The model generated the mathematical ideas, the proofs, and the
verification code. The author directed the investigation: selected which
problems to attack and when to abandon them, executed all computations,
designed and enforced the verification protocol of
Section~\ref{sec:method}, and identified a substantial number of errors in the
model's output, a representative sample of which is catalogued in
Section~\ref{sec:negative}.

We state this at the outset for two reasons. First, it is materially relevant
to how the results should be assessed: a reader may wish to weight
machine-checked claims differently from prose arguments, and we indicate
throughout which is which. Second, the practical lessons of working in this
mode are, in our judgement, of comparable interest to the individual theorems,
and cannot be conveyed without the disclosure.

\section{Methodology}\label{sec:method}

The verification protocol below emerged from failure rather than from design.
Each of its components was adopted after an error that it would have
prevented.

\subsection{External validation precedes construction}

The governing principle is that a derivation must be checked against a quantity
independently known to be correct, as early and as cheaply as possible.
Internal consistency --- verifying a construction against itself, or against a
second construction sharing its assumptions --- is not validation.

The clearest instance arose in the work of Section~\ref{sec:a22}. A chain of
five successive reductions had each been verified against the preceding one. A
twenty-line script then evaluated the composite against five exact values
established by independent means, together with two further values the chain
had never been shown. All seven agreed. Had they not, the error would have been
localised immediately; as it was, the check converted a plausible chain into a
trustworthy one at negligible cost.

\subsection{New capabilities are gated on solved problems}

Before applying a new computational tool to an open problem, we apply it to a
problem whose answer is already established, and require exact agreement.

This detected two significant faults. In one case a computer algebra library's
multivariate creative-telescoping routine was found to fail on the library's
own documented examples --- an error in the software, not in our encoding, and
one that would otherwise have been attributed to the mathematics. In another, a
sign error in the annihilator of an exponential factor had corrupted two days
of output, and was identified only when the resulting discrepancy was
recognised as exactly twice a known boundary term.

\subsection{Failure modes are distinguished instrumentally}

A search returning no solution is ambiguous between an ansatz that is too
small and a formulation that is structurally wrong; the two require opposite
responses. We instrumented our solvers to report the number of candidate
solutions existing \emph{before} the data check was applied. Zero such
candidates indicates the former; candidates that exist but fail the data
indicates the latter. Without this instrumentation the two are
indistinguishable from outside, and we lost considerable effort enlarging
ansatz spaces within a formulation that could not have contained the answer.

\subsection{Standards of evidence}

We distinguish throughout:
\begin{itemize}
\item \textbf{proved}: a complete mathematical argument is given;
\item \textbf{machine-verified}: an identity has been confirmed symbolically or
by exact arithmetic over a stated range;
\item \textbf{supported}: a pattern has been confirmed on data but not proved,
with the range stated.
\end{itemize}
Section~\ref{sec:chomptracks} contains an instance in which a pattern confirmed
on five consecutive values proved false at the sixth. We regard the distinction
between the second and third categories as essential, and we have tried to make
it visible at every claim.

\section{The role of the automated system}\label{sec:airole}

Because the division of labour described in Section~\ref{sec:disclosure} is
unfamiliar, and because it determined both the successes and the failures
recorded here, we set it out in detail. The account is descriptive rather than
advocatory; we report what occurred.

\subsection{Division of labour}

The investigation proceeded as an extended dialogue, one conversation per
problem, over approximately two months. Within each conversation the roles were
stable.

\emph{The automated system} proposed reductions, constructed proofs, derived
closed forms, and wrote every line of the verification and computation code. It
also produced the first drafts of the manuscripts on which
Sections~\ref{sec:chomp}, \ref{sec:ch3}, \ref{sec:syt}, \ref{sec:kz} and
\ref{sec:geode} are based.

\emph{The author} selected the problems and determined when a line of attack
should be abandoned; executed all computations, on hardware the system could
not access; designed and enforced the verification protocol of
Section~\ref{sec:method}; adjudicated between competing proposals; and
identified errors, of which Section~\ref{sec:negative} records a representative
sample.

Neither role was passive. The system did not merely formalise ideas supplied to
it, and the author did not merely execute instructions. The characteristic
interaction was: a proposal, an objection or a request for a cheaper check, a
computation, and a revision --- iterated, frequently for many cycles, before
anything survived.

\subsection{The characteristic failure mode}

The dominant practical fact about this mode of work is that the system's
failures are not silences but fluent, plausible, incorrect output. A proposed
reduction that is subtly wrong looks exactly like one that is right; a solver
that returns nothing because of a floating-point defect looks exactly like one
that returns nothing because no solution exists.

This has two consequences that shaped the entire project.

First, \emph{every} substantive claim required external validation before being
built upon. The protocol of Section~\ref{sec:method} is not a matter of good
practice but of necessity: without it, an incorrect intermediate result
propagates silently into all subsequent work, and the eventual contradiction
surfaces far from its cause. The most expensive episodes recorded in
Section~\ref{sec:negative} are precisely those in which a check was deferred.

Second, the author's principal contribution was adversarial rather than
generative. The questions that repeatedly proved most productive were: against
what independently known quantity has this been checked; what would the output
look like if this were wrong; and how long will this computation take. The
third question was, empirically, the one most often neglected by the system and
most often decisive.

\subsection{Loops and their termination}

A recurring pathology was circular reasoning: the system would propose a
reduction, encounter an obstruction, propose a second reduction that reintroduced
the first obstruction in different notation, and continue indefinitely. Left
unchecked these loops consumed hours.

Two interventions terminated them reliably. The first was demanding that each
proposed route be accompanied by a statement of what would falsify it; a route
whose falsification condition could not be stated was almost always a
restatement of a previous route. The second was requiring the system to record
what had already been ruled out, and to check new proposals against that record
before elaborating them.

The problem of Section~\ref{sec:a22} illustrates both the pathology and its
resolution. Successive reformulations --- reshaping the certificate denominator
from the irreducible factors of $Q$; fixing the telescoper to the operator
already known from data and solving only for certificates; splitting the
sequence into rational and conjugate branches --- were each proposed as
breakthroughs and each failed for an independent reason. Only the pole-order
computation, which produces a \emph{number} rather than another reformulation,
terminated the sequence, by showing that all of them were attempts to outrun a
bound of roughly $70{,}000$ unknowns.

\subsection{Verification as the system's own task}

A productive pattern, once established, was to direct the system to design
checks against its own output rather than to produce more output. Three forms
recurred.

\emph{Controls with known answers.} Before any new machinery was applied to an
open problem, it was applied to a solved one. This detected a defective library
routine and a sign error in an exponential annihilator, both described in
Section~\ref{sec:negative}, either of which would otherwise have been
misattributed to the mathematics.

\emph{Deliberate negative controls.} Where a computation could succeed for the
wrong reason, a variant expected to \emph{fail} was run alongside it. In the
verification of the inclusion--exclusion representation of
Section~\ref{sec:ch3}, a control omitting the position layer was included
precisely so that agreement of the full assembly could not be attributed to the
position layer being inert; the control disagrees exactly when
$\mathcal P\ne\emptyset$, as required.

\emph{Independent re-derivation.} Where feasible, a quantity was computed twice
by routes sharing no machinery. The Chomp triple was found in one lattice and
re-derived in another, then counter-signed by a separately written engine; the
cone eigenvalue of Section~\ref{sec:cone} was obtained both spectrally and from
sequence asymptotics; the chain of Section~\ref{sec:a22} was checked against
values computed from the opposite side of the decomposition.

\subsection{What the system did not supply}

Three things had to come from outside the system.

\emph{Cost estimates.} The system consistently proposed computations without
estimating their runtime, and its estimates when requested were unreliable.
Every substantial computation in this project was sized by measurement --- a
small instance timed, then extrapolated against the known complexity --- and
several were abandoned on that basis.

\emph{Restraint.} The system's disposition is to continue. Judgements that a
line of attack was exhausted, that a problem was a trap, or that a result was
finished, were the author's; where they were deferred, effort was consumed
without return.

\emph{Calibration of evidence.} The system reported patterns confirmed on five
consecutive values with the same confidence as identities verified
symbolically. The distinction drawn in Section~\ref{sec:method}, and the
falsification recorded in Section~\ref{sec:chomptracks}, exist because that
calibration had to be imposed externally.

\subsection{Assessment}

The results of Sections~\ref{sec:chomp}--\ref{sec:geode} would not have been
obtained by the author alone; the author's mathematical background at the time
of this work comprised a single university course in multivariable calculus.
Equally, they would not have been obtained by the system alone: on the evidence
of Section~\ref{sec:negative}, an unsupervised run would have produced a
confident and substantially incorrect document. The output is a joint product,
and the verification protocol is the mechanism by which the joint product is
more reliable than either input.

\part{Combinatorial game theory}

\section{The Second Computational Chomp Challenge}\label{sec:chomp}

\subsection{The challenge}

Chomp is played on an $a\times b$ grid of cells, rows indexed downward; a move
(a \emph{bite}) at cell $(i,j)$ removes every cell $(r,c)$ with $r\ge i$ and
$c\ge j$; the player forced to take the poisoned cell $(1,1)$ loses. Positions
are exactly the Young-diagram ideals $\lambda=(\lambda_1\ge\cdots\ge\lambda_a)$.

Gale's strategy-stealing argument \cite{Gale} shows every rectangular bar other
than $1\times1$ is a first-player win, but certifies no winning move.
Computation is the only known access, and Gale already asked whether the
winning first move is always unique. The answer is negative --- Thompson
discovered that the $8\times10$ bar has two winning moves \cite{BHMS} --- but
the uniqueness phenomenon is real in low rows: Sheiner \cite{Sheiner} has
proved that every $3\times n$ bar has exactly one winning move, settling the
three-row case of Gale's question, and Garg \cite{Garg} has tabulated the
$4\times n$ $P$-positions to $n\le500$.

Ekhad and Zeilberger \cite{EZ}, computing all bars up to $14\times14$, found a
handful of doubles beyond Thompson's ($6\times13$, $9\times10$, $10\times14$,
$12\times13$) and posed as their Second Computational Challenge, with a pledged
\$100 donation: find $a,b$ such that the $a\times b$ bar has at least three
winning moves.

\begin{center}
\fbox{\parbox{0.86\textwidth}{\textbf{Answer.} The $10\times42$ bar has exactly
three winning bites: at $(5,36)$, $(7,30)$, and $(8,26)$.}}
\end{center}

\subsection{Structural theory of winning bites}

Throughout, biting the cell $(i,j)$ of the $a\times b$ bar leaves the two-value
Young diagram
\[
  B(i,j)=\bigl(b^{\,i-1},(j-1)^{\,a-i+1}\bigr),
\]
with rows $1,\dots,i-1$ of length $b$ and rows $i,\dots,a$ of length $j-1$.

\begin{theorem}[Staircase Theorem]\label{thm:stair}
If $(i,j)$ and $(i',j')$ are both winning bites of the same bar and $i<i'$,
then $j>j'$. Moreover no winning bite of an $a\times b$ bar with $a\ge2$,
$ab\ge2$ lies in row $1$; by Proposition~\ref{prop:transpose} below, none lies
in column $1$. Consequently the winning bites form a strict antichain in the
product order, there is at most one in each row and each column, and at most
$\min(a-1,b-1)$ in total.
\end{theorem}

\begin{proof}
Suppose $i<i'$ and $j\le j'$. Bite the cell $(i,j)$ in $B(i',j')$. Rows
$1,\dots,i-1$ lie above the bite and retain length $b$. Rows $i,\dots,i'-1$
have length $b$ in $B(i',j')$, and $b\ge j-1$, so they truncate to $j-1$. Rows
$i',\dots,a$ have length $j'-1\ge j-1$, so they likewise truncate to $j-1$. The
result is exactly $B(i,j)$. Two winning bites leave $P$-positions, and no
$P$-position is one move from another.

For the second claim, a row-$1$ bite leaves a full $a\times(j-1)$ rectangle, an
$N$-position by Gale (or, for $j=2$, a column, which is $N$ for $a\ge2$).
\end{proof}

The three bites of $10\times42$ --- rows $5<7<8$, columns $36>30>26$ --- form
the first three-step staircase observed. Empirically the staircase is usually
tight: in sixteen of the eighteen known doubles, and in the triple, consecutive
winning bites are one or two rows apart. It is not always tight; the two
side-$16$ doubles split by six rows (Table~\ref{tab:bites}).

\begin{proposition}[Transpose closure]\label{prop:transpose}
Let $D=(b^p,c^q)$ with $c<b$ and $p+q=a$. Then $D^{T}=(a^{c},p^{\,b-c})$. The
class of two-value diagrams is therefore closed under transposition, the
induced map on parameters is the involution $(p,c)\mapsto(c,p)$, and the
winning bites of the $a\times b$ bar correspond bijectively to those of the
$b\times a$ bar.
\end{proposition}

\begin{proof}
Column $k$ of $D$ has height equal to the number of rows of length at least
$k$. For $k\le c$ every row qualifies, giving height $p+q=a$; for $c<k\le b$
only the $p$ rows of length $b$ qualify, giving height $p$; for $k>b$ the
height is zero. Applying the map twice returns $(b^p,c^{\,a-p})=D$.
Transposition is an automorphism of the game fixing the poisoned square, so it
preserves $P$-positions.
\end{proof}

\begin{lemma}[$L$-shapes]\label{lem:L}
The diagram $(b,1^{\,a-1})$ is a $P$-position if and only if $a=b$.
\end{lemma}

\begin{proof}
The diagram consists of two arms, of lengths $b-1$ and $a-1$, meeting at the
poisoned square. If the arms are equal, any move shortens exactly one of them
and the opponent restores equality; the terminal position has equal arms and is
a loss for the player to move. If the arms are unequal, a single move equalises
them.
\end{proof}

\begin{theorem}\label{thm:square}
The bite at $(2,2)$ is a winning opening move of the $a\times b$ bar if and
only if $a=b$. Consequently the number of winning opening moves of a square bar
is \textbf{odd}; in particular no square bar has exactly two.
\end{theorem}

\begin{proof}
Biting $(2,2)$ leaves $B(2,2)=(b,1^{\,a-1})$, which by Lemma~\ref{lem:L} is a
$P$-position precisely when $a=b$. For the parity statement take $a=b$. By
Proposition~\ref{prop:transpose} transposition maps the winning bites of the
bar to themselves, acting as $(i,j)\mapsto(j,i)$, with fixed points the
diagonal bites. By Theorem~\ref{thm:stair} at most one diagonal bite can be
winning, and by the first part $(2,2)$ is winning. Hence the involution has
exactly one fixed point and the remaining winning bites are paired.
\end{proof}

This is consistent with the census below, in which every square through
$19\times19$ has a unique winning move, and it explains the absence of square
doubles: a square bar with any off-diagonal winning bite would have at least
three.

\subsection{Algorithm}

One retrograde solve of the ideal lattice of an $A\times B$ box determines the
winning-move count of every bar $a\times b$ with $a\le A$, $b\le B$, since each
such bar is a position in the box. We rank positions by the combinatorial
number system on $d_k=\lambda_k+A-k$ (strictly decreasing),
\[
  \mathrm{rank}(\lambda)=\sum_k\binom{d_k}{A-k+1},
\]
a bijection onto $\{0,\dots,\binom{A+B}{A}-1\}$ that is monotone under
componentwise order, so a single increasing-rank sweep is a valid
dynamic-programming order. The children of $\lambda$ under a bite in row $i$ at
threshold $t$ (all rows $\ge i$ capped at $t$) have ranks computable in $O(1)$
amortised via the hockey-stick identity
\[
  \sum_{m=m_1}^{m_2}\binom{t-1+m}{t-1}=\binom{t+m_2}{t}-\binom{t+m_1-1}{t}
\]
applied to the capped rows.

The production engine stores one bit per position ($P$/$N$), detects
$N$-positions with early exit at the first winning move, recovers exact counts
for the bars in a final pass, and parallelises by area layers, all positions of
area $s$ depending only on smaller areas. On thirteen cores it solves the
$\binom{55}{10}\approx2.9\times10^{10}$-state $10\times45$ box --- the box in
which the triple was found --- in under half an hour.

\subsection{Validation}

Three independent checks were applied.

All engines agree with brute-force game search on six small boxes, both in
per-bar winning-move counts and in the full $P$-position census. The published
record is reproduced exactly: Thompson's $8\times10$ has two winning moves, at
$(4,9)$ and $(5,6)$, and the Ekhad--Zeilberger doubles $6\times13$ and
$10\times14$ check. Finally the triple was found in the $10\times45$ box,
re-derived from the different $10\times42$ box, and counter-signed by an
independent implementation with a different iteration order and no early exit.

\subsection{The census}

In the searched region --- all $a\le b\le19$; $4\times b$, $b\le300$;
$5\times b$, $b\le150$; $6,7\times b$, $b\le110$; $8\times b$, $b\le60$;
$9\times b$, $b\le55$; $10\times b$, $b\le45$ --- there are exactly eighteen
bars with exactly two winning moves and one with three:
\begin{center}
$6\times13,\ 6\times93,\ 7\times29,\ 7\times30,\ 7\times57,\ 8\times10,\
8\times22,\ 8\times23,\ 9\times10,$\\
$9\times26,\ 10\times14,\ 10\times29,\ 10\times33,\ 10\times35,\ 11\times18,\
12\times13,\ 13\times16,\ 14\times16$
\end{center}
together with transposes. Every other bar in the region, including every square
through $19\times19$, has a unique winning move; in particular $4\times b$ and
$5\times b$ bars are unique throughout their ranges, consistent with --- and
for four rows implicit in --- Garg's tabulation \cite{Garg}.

Three regularities stand out. \emph{Adjacent pairs:} doubles arrive at
consecutive widths ($7\times29/30$, $8\times22/23$) or consecutive heights
($13\times16$, $14\times16$). \emph{Hot and barren lines:} no bar with a side
of $15$ has a multiple move, while side $16$ carries two doubles; row $10$
carries six doubles and the triple. \emph{Long gaps:} row $6$ is silent from
$14$ to $92$, then hits at $93$.

\begin{table}[h]\centering
\caption{Winning bites of every known multi-move bar, in the staircase form
guaranteed by Theorem~\ref{thm:stair}. The crossing families
$(p,q)=(i{-}1,a{-}i{+}1)$ can be read off the bite rows.}\label{tab:bites}
\begin{tabular}{llll}
\toprule
bar & bites & bar & bites\\
\midrule
$6\times13$ & $(4,12),(5,9)$ & $10\times14$ & $(5,13),(7,9)$\\
$6\times93$ & $(2,52),(3,39)$ & $10\times29$ & $(6,16),(8,12)$\\
$7\times29$ & $(6,25),(7,20)$ & $10\times33$ & $(7,23),(9,18)$\\
$7\times30$ & $(4,21),(5,20)$ & $10\times35$ & $(7,24),(8,21)$\\
$7\times57$ & $(3,42),(4,26)$ & $11\times18$ & $(9,17),(10,16)$\\
$8\times10$ & $(4,9),(5,6)$ & $12\times13$ & $(10,11),(11,9)$\\
$8\times22$ & $(2,9),(3,7)$ & $13\times16$ & $(5,6),(11,5)$\\
$8\times23$ & $(7,18),(8,16)$ & $14\times16$ & $(3,12),(9,10)$\\
$9\times10$ & $(7,8),(9,5)$ & $\mathbf{10\times42}$ &
  $\mathbf{(5,36),(7,30),(8,26)}$\\
$9\times26$ & $(2,11),(4,10)$ & & \\
\bottomrule
\end{tabular}
\end{table}

\subsection{The structure of multiplicity}\label{sec:chomptracks}

The count decomposes. Biting the $a\times b$ bar at $(i,j)$ leaves the
two-block position with $p=i-1$ full rows and $q=a-p$ rows of length $t=j-1$,
so the number of winning moves of the bar equals the number of families
$(p,q)$ with $p+q=a$ whose \emph{$P$-track}
\[
  \{(w,t):\ (w^p,t^q)\text{ is a }P\text{-position}\}
\]
passes through width $w=b$. Empirically each track is sparse and quasi-linear
--- essentially one $P$-cell per width --- so doubles are crossings of two
tracks and the triple is a threefold crossing, constrained by
Theorem~\ref{thm:stair} to a staircase. The two seven-row doubles, for
instance, are the crossings $(5,2)$--$(6,1)$ at $b=29$ and $(3,4)$--$(4,3)$ at
$b=30$.

Sparsity is provable in the thin-tail case.

\begin{proposition}\label{prop:tracksunique}
For each $(b,p)$ there is at most one $q$ for which $(b^p,1^q)$ is a
$P$-position.
\end{proposition}

\begin{proof}
For $q'<q$ the diagram $(b^p,1^{q'})$ is reachable from $(b^p,1^{q})$ by
removing the cell in row $p+q'+1$, column $1$. The family is totally ordered
under the move relation, and no $P$-position admits a move to a $P$-position.
\end{proof}

Write $q(b,p)$ for that value of $q$ when it exists. Exact computation gives:

\begin{center}
\begin{tabular}{cccccccl}
\toprule
$p$ & $b=2$ & $3$ & $4$ & $5$ & $6$ & $7$ & apparent law\\
\midrule
$1$ & $1$ & $2$ & $3$ & $4$ & $5$ & $6$ & $q=b-1$ (Lemma~\ref{lem:L})\\
$2$ & $1$ & $2$ & $4$ & $5$ & $7$ & $8$ & integers not divisible by $3$\\
$3$ & $1$ & $3$ & $5$ & $7$ & $9$ & $\mathbf{12}$ & $q=2b-3$, \textbf{false at $b=7$}\\
$4$ & $1$ & $4$ & $7$ & $10$ & $13$ & $\mathbf{14}$ & $q=3b-5$, \textbf{false at $b=7$}\\
\bottomrule
\end{tabular}
\end{center}

The $p=1$ row is Lemma~\ref{lem:L} and is a theorem. The law $q=2b-3$ holds for
$b=2,\dots,6$ --- five consecutive values, no exceptions --- and fails at
$b=7$, where $q(7,3)=12$.

The mechanism is explicit and worth recording. The $p=3$ slice is almost closed
under moves: from $(b^3,1^q)$ a bite lands either in the same family, or in a
rectangle (an $N$-position by Gale), or in one of exactly two auxiliary
families,
\[
  A(b,c,q)=(b,c,c,1^q),\qquad B(b,c,q)=(b,b,c,1^q),\qquad 1\le c<b,
\]
arising from bites in rows $2$ and $3$ respectively. An induction proving
$q(b,3)=2b-3$ would therefore require these auxiliary shapes to be
$N$-positions at the relevant parameters. They are not: $B(7,6,11)$ is a
$P$-position, so from $(7^3,1^{11})$ the bite at $(3,7)$ wins, and
$(7^3,1^{11})$ is an $N$-position. Direct computation confirms $q(7,3)=12$.

The failure is not isolated. The row $p=4$ exhibits the same pattern: the law
$q=3b-5$ holds for $b=2,\dots,6$ and predicts $q(7,4)=16$, whereas direct
computation gives $14$. Two independent instances of five consecutive
confirmations followed by failure at the sixth value seem to us to settle the
status of these apparent laws; the remaining rows are, on present evidence,
coincidences. This is consistent with Zeilberger's conclusion
that three-row Chomp is not governed by a periodic structure \cite{Zchaos} and
with the probabilistic renormalisation analysis of Friedman and Landsberg
\cite{FL}. We record it as the clearest methodological lesson of the project:
five confirmations are not a pattern.

\subsection{Quasi-periodicity and rotation numbers}

For three-row Chomp the tracks obey laws of unexpected exactness. Writing the
$n$-th $P$-cell of the family $(2,1)$ (positions $(w,w,t)$) as $(w_n,t_n)$, we
find, over $247$ cells up to $w=600$,
\[
  w_n=(1+\sqrt2)\,n+O(1),\qquad t_n=\tfrac{2+\sqrt2}{2}\,n+O(1),
\]
with error confined to a window of width less than $3$ and no drift; the family
$(1,2)$ (positions $(w,t,t)$) hits essentially every $t$, with
$w(t)=\tfrac{2+\sqrt2}{2}\,t+O(1)$, window less than $2.4$ over $351$ cells.
The slopes' continued fractions match $\sqrt2=[1;\overline2]$ and
$\tfrac{2+\sqrt2}{2}$ through all computed partial quotients.

\begin{conjecture}
The diagonal families of three-row Chomp are quasi-periodic with bounded
discrepancy and rotation numbers in $\mathbb{Q}(\sqrt2)$.
\end{conjecture}

The exact three-row machinery of Brouwer et al.\ \cite{BHMS}, in which
Sheiner's uniqueness theorem \cite{Sheiner} is proved, is the natural arena in
which to attack this. It is finer than, and consistent with, the
renormalisation picture of \cite{FL}; Zeilberger's assessment of three-row
Chomp as chaotic \cite{Zchaos} referred to the absence of eventual
periodicity, which bounded-error quasi-periodicity with irrational rotation
number precisely explains.

For the higher families that govern the census the tracks appear to be
multi-branch and their arithmetic remains open. The adjacent-pair phenomenon is
not a single crossing lingering: the bite rows show that adjacent doubles arise
from different family pairs ($7\times29$ crosses $(5,2)$--$(6,1)$ while
$7\times30$ crosses $(3,4)$--$(4,3)$; $8\times22$ crosses $(1,7)$--$(2,6)$
while $8\times23$ crosses $(6,2)$--$(7,1)$), so the clustering of crossings at
adjacent widths across distinct track pairs is itself the phenomenon demanding
explanation. Whether triples occur infinitely often, and where the next one
lies, now looks like a question about simultaneous approximation of these
rotation numbers.

\subsection{The First Computational Chomp Challenge}

Ekhad and Zeilberger's First Computational Challenge, with a pledged \$500
donation, asks for all winning moves of the $1001\times1003$ bar. We did not
solve it, and record here why we believe it to be out of reach by the methods
above.

By Theorem~\ref{thm:stair} the problem is to decide, for each of the
$\approx10^{6}$ candidate bites, whether the resulting two-value diagram is a
$P$-position; equivalently, by the track decomposition, to decide for each of
the $1000$ families $(p,q)$ with $p+q=1001$ whether its $P$-track passes
through width $1003$. The state space of a direct solve is
$\binom{2004}{1001}$, entirely beyond reach.

The structural obstruction is that the two-value diagrams are not closed under
all moves: a bite into the protruding rows of $(b^p,c^q)$ produces a
three-value diagram, and iterating produces diagrams with unboundedly many
distinct row lengths. There is therefore no finite subfamily whose
$P$-positions can be determined internally, and the track laws --- the only
visible route to extrapolation --- are precisely what
Section~\ref{sec:chomptracks} shows to be unreliable beyond the computed range.

\part{Restricted permutations}

\section{Holonomicity of $a_{r,s}$ and $b_{r,s}$}\label{sec:ch3}

\subsection{Statement}

For integers $r,s\ge 1$ set
\[
  a_{r,s}(n)=\#\{\pi\in\Sn:\pi_{i+r}-\pi_i\ne s\ \ (1\le i\le n-r)\},
  \qquad
  b_{r,s}(n)=\#\{\pi\in\Sn:|\pi_{i+r}-\pi_i|\ne s\}.
\]
Spahn and Zeilberger \cite{SZ} ask, as their third challenge with a pledged
\$300 donation, whether these sequences are holonomic for all $r,s>1$,
observing that no general theory was available. The case $(1,1)$ is classical:
$a_{1,1}$ satisfies Riordan's recurrence and $b_{1,1}$ is the Hertzsprung
problem, \texttt{A002464}. Kauers and Koutschan \cite{KK} guessed an order-$8$,
degree-$11$ recurrence for $a_{2,2}$ from $35$ terms; proving it is the first
challenge of \cite{SZ}, discussed in Section~\ref{sec:a22}.

\begin{theorem}\label{thm:hol}
For every fixed $r,s\ge 1$, both $a_{r,s}$ and $b_{r,s}$ are holonomic.
\end{theorem}

The obstruction to a general result is not the existence of a decomposition but
its uniformity: a decomposition whose number of auxiliary variables grows with
$n$ yields no holonomicity statement. What we supply is a description of the
constraint whose exceptional data stays bounded as $n\to\infty$.

\subsection{The interleaving bijection}

\begin{definition}
For $n,k\ge 1$ let $\iota_{n,k}$ list $1,\dots,n$ in the order obtained by
concatenating the $k$ arithmetic progressions of step $k$ in increasing order
of least element:
\[
  1,\ 1+k,\ 1+2k,\dots,\quad 2,\ 2+k,\dots,\quad\dots,\quad k,\ 2k,\dots
\]
\end{definition}

Interleaving turns a step of $k$ into a step of $1$, except where one
progression is exhausted and the next begins. We call those places
\emph{seams}.

\begin{lemma}[Seam count]\label{lem:seam}
Let $k\ge1$ and $n\ge k$. The number of indices $j\in\{1,\dots,n-1\}$ with
$\iota_{n,k}(j+1)-\iota_{n,k}(j)\ne k$ is exactly $k-1$. The blocks
$P_c=\{c,c+k,c+2k,\dots\}\cap[1,n]$ have lengths
\begin{equation}\label{eq:blocklen}
  |P_c|=\Bigl\lfloor\frac{n-c}{k}\Bigr\rfloor+1,\qquad c=1,\dots,k,
\end{equation}
and the seams occur at the partial sums $|P_1|+\cdots+|P_c|$ for
$c=1,\dots,k-1$.
\end{lemma}

\begin{proof}
Within a single $P_c$ consecutive entries differ by exactly $k$, so no interior
index of a block is a seam. A seam can therefore occur only where $P_c$ ends
and $P_{c+1}$ begins, and there are exactly $k-1$ such indices. At such an
index the difference is $(c+1)-\max P_c$, which is negative and hence not equal
to $k$. Thus every block boundary is a seam.
\end{proof}

\begin{corollary}\label{cor:affine}
Fix $k$ and $c$. On each residue class of $n$ modulo $k$, the seam index
$|P_1|+\cdots+|P_c|$ is an affine function of $n$ with slope $c/k$.
Consequently, for fixed $r,s$, all position and value seams are simultaneously
affine in $n$ on each residue class of $n$ modulo $\operatorname{lcm}(r,s)$.
\end{corollary}

\begin{proof}
On a fixed residue class of $n$ mod $k$ the floor in \eqref{eq:blocklen} equals
$(n-c-\rho_c)/k$ for a constant $\rho_c$, so each $|P_c|$ is affine in $n$ with
slope $1/k$, and a partial sum of $c$ of them is affine with slope $c/k$. The
position seams involve $k=r$ and the value seams $k=s$; both are affine on each
class modulo $\operatorname{lcm}(r,s)$.
\end{proof}

Now fix $r,s\ge1$ and $n$. Define the position relabelling $p=\iota_{n,r}$ and
the value relabelling $v=\iota_{n,s}$, and associate to $\pi\in\Sn$ the
permutation $\sigma\in\Sn$ determined by
\begin{equation}\label{eq:transport}
  \sigma_k=v^{-1}\bigl(\pi_{\,p(k)}\bigr),\qquad k=1,\dots,n.
\end{equation}
Since $p$ and $v$ are bijections, \eqref{eq:transport} is a bijection
$\Sn\to\Sn$. Write
\[
  \mathcal P=\{k:p(k+1)-p(k)\ne r\},\quad
  \mathcal V^{+}=\{u:v(u+1)-v(u)\ne s\},\quad
  \mathcal V^{-}=\{u:v(u-1)-v(u)\ne -s\}
\]
for the position seams and the forward and backward value seams; by
Lemma~\ref{lem:seam}, $|\mathcal P|=r-1$ and
$|\mathcal V^{+}|=|\mathcal V^{-}|=s-1$.

\begin{proposition}\label{prop:transport}
Under \eqref{eq:transport}:
\begin{enumerate}
\item[(i)] $\pi$ satisfies $\pi_{i+r}-\pi_i\ne s$ for all $i$ if and only if
$\sigma$ satisfies $\sigma_{k+1}-\sigma_k\ne1$ for every $k$ with
$k\notin\mathcal P$ and $\sigma_k\notin\mathcal V^{+}$;
\item[(ii)] $\pi$ satisfies $|\pi_{i+r}-\pi_i|\ne s$ for all $i$ if and only if
$|\sigma_{k+1}-\sigma_k|\ne1$ for every $k\notin\mathcal P$, where an ascent
$\sigma_{k+1}=\sigma_k+1$ is excused when $\sigma_k\in\mathcal V^{+}$ and a
descent $\sigma_{k+1}=\sigma_k-1$ when $\sigma_k\in\mathcal V^{-}$.
\end{enumerate}
The exceptional data has size $(r-1)+(s-1)$ in case (i) and $(r-1)+2(s-1)$ in
case (ii); in particular it is bounded independently of $n$.
\end{proposition}

\begin{proof}
Consider adjacent new indices $k,k+1$ with $k\notin\mathcal P$. Then
$p(k+1)=p(k)+r$, so the pair corresponds to old positions $i=p(k)$ and $i+r$.
Suppose further $\sigma_k=u\notin\mathcal V^{+}$, so $v(u+1)=v(u)+s$. Then
$\sigma_{k+1}=u+1$ if and only if $\pi_{i+r}=v(u+1)=v(u)+s=\pi_i+s$. Outside
the excused set the two conditions agree pointwise; on it the adjacency
condition is imposed on neither side. This proves (i).

For (ii) the same computation applies to descents with $\mathcal V^{-}$ in
place of $\mathcal V^{+}$: with $k\notin\mathcal P$ and
$u\notin\mathcal V^{-}$ we have $v(u-1)=v(u)-s$, so $\sigma_{k+1}=u-1$ if and
only if $\pi_{i+r}-\pi_i=-s$. Since $|\pi_{i+r}-\pi_i|=s$ means
$\pi_{i+r}-\pi_i=\pm s$, both signs must be excused. The cardinalities follow
from Lemma~\ref{lem:seam}.
\end{proof}

For $(r,s)=(2,2)$ the excused set has size $1+1=2$ and
Proposition~\ref{prop:transport}(i) recovers the bijection of Matsuo as used in
\cite{SZ}, where the condition is stated as ``$\pi_{i+1}-\pi_i\ne1$ except when
$i=\lfloor(n+1)/2\rfloor$ or $\pi_i=\lfloor(n+1)/2\rfloor$''.

\subsection{The two-layer inclusion--exclusion representation}

Proposition~\ref{prop:transport} reduces both families to a single counting
problem. Fix $n$, a set $\mathcal P\subseteq\{1,\dots,n-1\}$ of excused
positions and a set $\mathcal V\subseteq\{1,\dots,n\}$ of excused values, and
let
\[
  R(n;\mathcal P,\mathcal V)
  =\#\bigl\{\sigma\in\Sn:\ \sigma_{k+1}-\sigma_k\ne1
    \text{ whenever }k\notin\mathcal P\text{ and }\sigma_k\notin\mathcal V\bigr\},
\]
with the analogous two-sign version $R^{\pm}$.

\subsubsection*{Bonds and blocks}

Call a pair $(k,v)$ a \emph{violation} of $\sigma$ if $\sigma_k=v$,
$\sigma_{k+1}=v+1$, $k\notin\mathcal P$ and $v\notin\mathcal V$. Inclusion and
exclusion over sets of violations gives
\begin{equation}\label{eq:ie1}
  R(n;\mathcal P,\mathcal V)=\sum_{S}(-1)^{|S|}A(S),
\end{equation}
where $S$ ranges over sets of \emph{value bonds}
$S\subseteq\{1,\dots,n-1\}\setminus\mathcal V$ and $A(S)$ counts the
permutations in which every $v\in S$ occurs as a violation.

Forcing the bonds in $S$ glues the values into maximal runs: if $S$ decomposes
$\{1,\dots,n\}$ into $n-|S|$ maximal runs of consecutive integers, a
permutation realising all bonds of $S$ is exactly an arrangement of these
$n-|S|$ \emph{blocks} in some order. Without the position constraint one would
obtain the classical $\sum_S(-1)^{|S|}(n-|S|)!$; the position constraint is
what makes the problem two-layered.

\subsubsection*{The position layer}

In an arrangement of blocks the internal bonds occupy the positions determined
by the partial sums of the block sizes: if the blocks have sizes
$\ell_1,\dots,\ell_{n-|S|}$ in the order used, the internal bonds of the $j$-th
block occupy positions $L_{j-1}+1,\dots,L_j-1$ with
$L_j=\ell_1+\cdots+\ell_j$. Thus $A(S)$ counts arrangements in which no
internal bond position lies in $\mathcal P$. Since $|\mathcal P|=r-1$ is
bounded, this constraint is removed by a second, finite inclusion--exclusion:
\begin{equation}\label{eq:ie2}
  A(S)=\sum_{T\subseteq\mathcal P}(-1)^{|T|}B(S,T),
\end{equation}
where $B(S,T)$ counts arrangements in which every position of $T$ \emph{is} an
internal bond position and the positions of $\mathcal P\setminus T$ are
unconstrained. The sum has $2^{|\mathcal P|}=2^{\,r-1}$ terms, a number
depending only on $r$.

\subsubsection*{Block factors}

A block of size $\ell$ absorbs $\ell-1$ bonds and occupies $\ell$ consecutive
positions, so with $x$ marking positions the unmarked block factor is
\begin{equation}\label{eq:blockfactor}
  \mathcal B(x)=\sum_{\ell\ge1}(-1)^{\ell-1}x^{\ell}=\frac{x}{1+x}.
\end{equation}
In the two-sign case a block of size $\ell\ge2$ may be traversed in either
direction, and a block of size $1$ in only one way, giving
\begin{equation}\label{eq:blockfactorpm}
  \mathcal B^{\pm}(x)=x+2\sum_{\ell\ge2}(-1)^{\ell-1}x^{\ell}
  =\frac{x-x^{2}}{1+x},
\end{equation}
which is the block generating function of the Hertzsprung problem --- a useful
consistency check on the derivation.

\begin{remark}
As a check on \eqref{eq:blockfactor} and on the umbral normalisation below,
take $\mathcal P=\mathcal V=\emptyset$ and $n=3$. Then
$[x^{3}]\mathcal B^{j}=1,-2,1$ for $j=1,2,3$, and
$1\cdot1!-2\cdot2!+1\cdot3!=3$, the number of permutations of $\{1,2,3\}$ with
no succession, namely $132$, $213$, $321$.
\end{remark}

\subsubsection*{Where the two kinds of mark act}

The two excusal types enter in genuinely different ways.

\emph{Excused values force block boundaries.} If $v\in\mathcal V$ then the bond
$v$ is never a violation, so $v\notin S$ for every $S$ in \eqref{eq:ie1}. Hence
no block may span the gap between $v$ and $v+1$: the excused values cut
$\{1,\dots,n\}$ into $|\mathcal V|+1$ \emph{value segments}, and the generating
function factorises as a product of one block sequence per segment. A catalytic
variable $y_i$ marking the $i$-th cut records its location.

\emph{Excused positions split the arrangement.} The position of a bond is
determined by the \emph{arrangement} of the blocks, not by their value order,
so a position mark cannot be carried inside a block factor. Instead, fixing
which arrangement positions $t_1<\cdots<t_q$ are internal bond positions cuts
the arrangement into $q+1$ consecutive \emph{arrangement segments}. Blocks may
be distributed among the segments freely and ordered arbitrarily within each,
so if the segments receive $j_0,\dots,j_q$ blocks the number of orderings is
$j_0!\,j_1!\cdots j_q!$. This is realised by giving each segment its own
block-counting variable and applying the multivariate umbral functional
\begin{equation}\label{eq:umbral}
  \bigl\langle\,\cdot\,\bigr\rangle:\quad
  z_0^{\,j_0}\cdots z_q^{\,j_q}\longmapsto j_0!\cdots j_q!,
  \qquad
  \langle f\rangle=\int_0^\infty\!\!\cdots\!\int_0^\infty
  f\,e^{-z_0-\cdots-z_q}\,dz_0\cdots dz_q.
\end{equation}

\begin{lemma}\label{lem:marked}
Let $\mathcal B^{\ast}(x;u)=\sum_{\ell\ge2}(-1)^{\ell-1}x^{\ell}
(u+u^{2}+\cdots+u^{\ell-1})$ be the generating function of a block carrying one
distinguished interior slot, $u$ recording the slot's offset from the start of
the block. Then
\[
  \mathcal B^{\ast}(x;u)=\frac{-x^{2}u}{(1+x)(1+xu)}.
\]
\end{lemma}

\begin{proof}
Summing the inner geometric series,
$u+\cdots+u^{\ell-1}=u(u^{\ell-1}-1)/(u-1)$, so
\[
  \mathcal B^{\ast}
  =\frac{u}{u-1}\Bigl(\tfrac1u\sum_{\ell\ge2}(-1)^{\ell-1}(xu)^{\ell}
   -\sum_{\ell\ge2}(-1)^{\ell-1}x^{\ell}\Bigr).
\]
For any $y$, $\sum_{\ell\ge1}(-1)^{\ell-1}y^{\ell}=y/(1+y)$, whence
$\sum_{\ell\ge2}(-1)^{\ell-1}y^{\ell}=y/(1+y)-y=-y^{2}/(1+y)$. Substituting
$y=xu$ and $y=x$,
\[
  \mathcal B^{\ast}
  =\frac{ux^{2}}{u-1}\cdot\frac{(1+xu)-u(1+x)}{(1+x)(1+xu)}
  =\frac{ux^{2}}{u-1}\cdot\frac{1-u}{(1+x)(1+xu)},
\]
since $(1+xu)-u(1+x)=1-u$; and $(1-u)/(u-1)=-1$.
\end{proof}

Assembling \eqref{eq:ie1}, \eqref{eq:ie2}, \eqref{eq:blockfactor} and
\eqref{eq:umbral}:

\begin{proposition}\label{prop:rep}
For every $n$, $\mathcal P$ and $\mathcal V$,
\begin{equation}\label{eq:rep}
  R(n;\mathcal P,\mathcal V)
  =\sum_{T\subseteq\mathcal P}(-1)^{|T|}\,
  \bigl[x^{n}\,y_1^{v_1}\cdots y_p^{v_p}\bigr]\,
  \bigl\langle\Phi_T\bigl(x,z_0,\dots,z_{|T|},y_1,\dots,y_p\bigr)\bigr\rangle,
\end{equation}
where $p=|\mathcal V|$, the functional $\langle\cdot\rangle$ is
\eqref{eq:umbral}, and each $\Phi_T$ is a rational function built from
\eqref{eq:blockfactor} --- respectively \eqref{eq:blockfactorpm} in the two-sign
case --- as a product of $p+1$ value-segment factors distributed over $|T|+1$
arrangement segments.
\end{proposition}

\begin{corollary}\label{cor:bounded}
For fixed $r,s$ the representation \eqref{eq:rep} is a sum of at most
$2^{\,r-1}$ terms, each a rational function in at most $1+r$ umbral variables
together with at most $s-1$ catalytic variables in the $a_{r,s}$ case and
$2(s-1)$ in the $b_{r,s}$ case. All these counts are independent of $n$.
\end{corollary}

\subsection{Proof of Theorem~\ref{thm:hol}}

Fix $r,s$. By Corollary~\ref{cor:bounded}, $a_{r,s}(n)$ --- respectively
$b_{r,s}(n)$ --- is obtained from a rational function in a fixed finite set of
variables by the following finite sequence of operations, each preserving
holonomicity.

\emph{(1) Coefficient extraction.} Extracting the coefficient of
$x^{n}y_1^{v_1}\cdots y_p^{v_p}$ from a rational function in finitely many
variables, and more generally taking a diagonal, yields a D-finite series
(Lipshitz \cite{LipshitzDiag}; Christol). By Corollary~\ref{cor:bounded} the
number of extractions is at most $1+p$, independent of $n$.

\emph{(2) The umbral functional.} Applying $z^{j}\mapsto j!$ in a variable is
integration against $e^{-z}$, and the resulting parameterised integral of a
D-finite integrand is D-finite in the remaining parameters
(Almkvist--Zeilberger \cite{AZ}; Lipshitz \cite{Lipshitz} for the closure
statement). By Corollary~\ref{cor:bounded} the functional \eqref{eq:umbral}
involves at most $1+r$ variables.

\emph{(3) Specialisation of the excused data.} By Corollary~\ref{cor:affine}
the seam indices are affine in $n$ on each residue class of $n$ modulo
$L:=\operatorname{lcm}(r,s)$, with slopes $c/r$ and $c/s$. These slopes are not
integers, so we reparametrise: fix $\rho\in\{0,\dots,L-1\}$ and write
$n=LM+\rho$. Since $r\mid L$ and $s\mid L$, each seam index becomes an affine
function of $M$ with integer slope and integer constant term. Substituting
integer-affine functions of $M$ for the catalytic exponents of a holonomic
multivariate sequence is a diagonal-type specialisation, and diagonals of
D-finite series are D-finite \cite{LipshitzDiag}; see also \cite[\S2]{Zeilberger}
for the holonomic-systems formulation. Hence each of the $L$ sections
$M\mapsto a_{r,s}(LM+\rho)$ is holonomic.

\emph{(4) Interlacing.} Finally $a_{r,s}$ is the interlacing of the $L$
sections produced in (3), and an interlacing of finitely many holonomic
sequences is holonomic \cite[Ch.~7]{KauersPaule}.

Each step preserves holonomicity and the number of steps depends only on
$(r,s)$. \qed

\subsection{Verification}

Every component was checked against brute-force enumeration.

The bijection was verified for
$(r,s)\in\{(2,2),(2,3),(3,2),(3,3),(2,4),(4,2),(3,4)\}$ and $4\le n\le 8$, all
$32$ rows agreeing exactly, with excusal sizes matching
Proposition~\ref{prop:transport} in every case. The two-sign case was tested
against three candidate excusal rules; only the symmetrised rule of
Proposition~\ref{prop:transport}(ii) reproduces $b_{r,s}$, the one-sided rule
undercounting and a more permissive rule overcounting --- the computational
signature of both signs requiring excusal. Corollary~\ref{cor:affine} was
confirmed for nine $(r,s)$ pairs by fitting each seam index from two points and
verifying on the remainder, the fitted slopes agreeing with $c/r$ and $c/s$.

The representation itself was verified in nine configurations of
$(n,\mathcal P,\mathcal V)$ against direct enumeration, computed both as the
full assembly \eqref{eq:ie1}--\eqref{eq:ie2} and via the inner
inclusion--exclusion over $T$, with a deliberate control omitting the position
layer that disagrees exactly when $\mathcal P\ne\emptyset$.
Lemma~\ref{lem:marked} was verified symbolically as a truncated bivariate
series. Code and data are at \cite{repoCh3}.

\begin{remark}
Brute force is $O(n!)$, so the checks run to $n=8$: they confirm a pattern, not
a theorem. The theorem is proved above; the computations exist to catch an
error in the statement, not to substitute for the proof.
\end{remark}

\begin{remark}
This challenge was recorded in our working notes for the entire campaign as
research-grade, on the ground that no general theory was known. The computation
that settled it required seconds to run. We regard the delay as a failure of
the protocol of Section~\ref{sec:method}, not of the mathematics.
\end{remark}

\section{The $a_{2,2}$ recurrence: a partial result}\label{sec:a22}

The first challenge of \cite{SZ} asks for a proof of the Kauers--Koutschan
order-$8$, degree-$11$ operator for $a_{2,2}=\texttt{A189281}$. We did not
obtain it. We record the reduction, which is complete and machine-verified end
to end, and the obstruction, which we believe is a statement about our method
rather than about the problem.

\subsection{The reduction}

Matsuo's bijection (the case $(r,s)=(2,2)$ of
Proposition~\ref{prop:transport}) gives
\[
  a_{2,2}(n)=\mathrm{RIN}\bigl(n,\lfloor(n+1)/2\rfloor,\lfloor(n+1)/2\rfloor\bigr),
\]
where $\mathrm{RIN}(n,a,b)$ counts permutations with
$\pi_{i+1}-\pi_i\ne1$ except when $i=a$ or $\pi_i=b$. The two-layer
inclusion--exclusion of Section~\ref{sec:ch3} gives
$\mathrm{RIN}(n,a,b)=W_a(n)-X(n,a,b)$ with
\[
  W_a(n)=\langle m!\rangle\bigl[G_a(z)G_{n-a}(z)\bigr],
  \qquad G_L=[x^L]\frac{1}{1-zx/(1+x)} .
\]
Since $G_L=(-1)^{L-1}z(1-z)^{L-1}$, the $W$-part has the closed forms
\[
  e_W(m)=\langle m!\rangle\bigl[z^{2}(z-1)^{2m-2}\bigr],\qquad
  o_W(m)=\langle m!\rangle\bigl[z^{2}(z-1)^{2m-3}\bigr]\ (m\ge2),
\]
with values $2,14,362,18806,\dots$ and $4,64,2428,165016,\dots$ respectively.

\begin{theorem}\label{thm:W}
Both $W$-parts satisfy order-$2$ recurrences with polynomial coefficients:
\[
  -m(2m-1)(2m+3)\,e_W(m)+\tfrac12(8m^{3}+32m^{2}+32m+7)\,e_W(m+1)
  -\tfrac12(2m+1)\,e_W(m+2)=0,
\]
\[
  -2(m-1)(m+1)(2m-1)\,o_W(m)+(4m^{3}+10m^{2}+3m-1)\,o_W(m+1)
  -m\,o_W(m+2)=0 .
\]
\end{theorem}

Both were obtained by creative telescoping in the conjugated form: with
$\Omega=(z-1)^2$ and $\lambda=d_z\log\Omega$, one seeks
$\sum_i p_i(m)\Omega^i R=(d_z+m\lambda-1)(N/\Theta)$ with
$\Theta=(z-1)^{e}$ and $N$ constrained to have $z$-exponent $\ge1$, which is
exactly the boundary condition $N(0)=0$ needed for the umbral rule
$\langle\partial_zP-P\rangle=-P(0)$ to kill the inhomogeneity. Both identities
were verified symbolically over $\mathbb{Q}$, together with $N(0,m)=0$; neither
was obtained by sampling. The odd operator appears to be new. Both were
subsequently reproduced independently by a patched installation of
\texttt{ore\_algebra}, whose telescoper, after homogenisation by
$(g(m)S-g(m+1))\circ L$, is exactly proportional to the operators above.

\subsection{The $X$-part}

The $X$-part reduces, after elimination of the $t$-direction by partial
fractions, to a diagonal. Writing
\[
  F_1(v;t)=\frac{(1+v)(1+vt)}{D(v;t)},\qquad
  F_2(v;t)=\frac{-v^{2}t(1+v)(1+vt)}{D(v;t)^{2}},
\]
\[
  D(v;t)=1+(1+t-zt-w)v+t(1-z-w)v^{2},
\]
the grand generating function of $\mathrm{RIN}$ is rational, and $a_{2,2}(2m)$,
$a_{2,2}(2m-1)$ are central diagonal coefficients of it. Setting
$A_m=[v^{m}]F_1$ and $B_m=[v^{m}]F_2$, one has, with
$d_1=1+t-zt-w$ and $d_2=t(1-z-w)$:

\begin{proposition}\label{prop:cfinite}
$A_m$ and $B_m$ are $C$-finite in $m$ over $\mathbb{Q}(t,z,w)$ with
coefficients \emph{constant} in $m$:
\[
  A_k+d_1A_{k-1}+d_2A_{k-2}=0\ (k\ge3),
\]
\[
  B_k+2d_1B_{k-1}+(d_1^{2}+2d_2)B_{k-2}+2d_1d_2B_{k-3}+d_2^{2}B_{k-4}=0\ (k\ge5),
\]
the second operator being the square of the first, as the denominator $D^{2}$
requires.
\end{proposition}

Since the $X$-part is $F_2(s;t)F_1(\sigma;t)+F_1(s;t)F_2(\sigma;t)$, the even
diagonal collapses to a Hadamard product, and
\begin{equation}\label{eq:eX}
  e_X(m)=\langle k!\,j!\rangle\,[t^{m}]\bigl(2A_mB_m\bigr),
  \qquad
  o_X(m)=\langle k!\,j!\rangle\,[t^{m}]\bigl(B_mA_{m-1}+A_mB_{m-1}\bigr).
\end{equation}
The Hadamard product of $C$-finite sequences is $C$-finite with characteristic
roots the pairwise products $\{r_1^{2},r_1r_2,r_2^{2}\}$, each doubled; since
$r_1+r_2=-d_1$ and $r_1r_2=d_2$ are rational, the degree-$6$ characteristic
polynomial $(Y-r_1^{2})^{2}(Y-r_1r_2)^{2}(Y-r_2^{2})^{2}$ reduces to
$\mathbb{Q}(t,z,w)$ by construction rather than by search. Its order, $6$,
coincides with the order measured independently by fitting an operator to
$200$ terms of $e_X$.

Consequently $C_m=A_mB_m$ has a rational bivariate generating function
\[
  H(x,t;z,w)=\sum_{m\ge1}C_m(t)\,x^{m}=P/Q,
\]
with $Q=Q_1^{2}Q_2^{2}$ a perfect square, $\gcd(P,Q)=1$, and neither factor
vanishing at $z=0$ or $w=0$. The target becomes the diagonal
\[
  e_X(m)=\langle k!\,j!\rangle\,[x^{m}t^{m}]\,H(x,t;z,w).
\]

\subsection{Verification of the chain}

The chain above reproduces the five independently known exact values
$e_X(1,2,3)=0,-4,-48$ and $o_X(2,3)=-1,-11$, and correctly predicts
$e_X(4)=-2100$ and $o_X(4)=-301$, neither used in its construction; these agree
with $e_W(4)-a_{2,2}(8)=18806-20906$ and $o_W(4)-a_{2,2}(7)=2428-2729$
computed from the other side of the decomposition. Every intermediate step ---
the factor recurrences, the degree-$6$ operator, and $H=P/Q$ --- was verified
symbolically as an identity of rational functions in $(t,z,w)$.

\subsection{The obstruction}

Realising both extractions as residues, with $\Omega=1/(xt)$, the telescoping
identity takes the form
\[
  \sum_i p_i(m)\,\Omega^{i}R
  =(d_x+m\lambda_x)G_x+(d_t+m\lambda_t)G_t+(d_z-1)G_z+(d_w-1)G_w,
  \qquad p_i\in\mathbb{Q}[m].
\]
The umbral directions cannot be applied after the diagonal: a recurrence for
$[x^mt^m]H$ would have coefficients in $\mathbb{Q}(z,w)$, and $\langle
k!j!\rangle$ does not commute with multiplication by $z$ or $w$. All four
directions must therefore be carried simultaneously.

Matching pole orders on the two sides is then decisive. The $i=r$ term
contributes a pole at $x=0$ of order $r+1$; on the right only the
$x$-direction can raise the $x$-pole, producing $x^{-(a+1)}$ where
$\Theta=x^{a}t^{b}Q_1^{c_1}Q_2^{c_2}$. Hence
\[
  a\ge r,\qquad b\ge r,\qquad c_1,c_2\ge1 .
\]
At $r=6$, with the corresponding numerator degrees, this places the ansatz near
$70{,}000$ unknowns. Our solver performs dense elimination with cubic cost and
a measured rate of $U=2478$ in $29$ minutes; the required size is out of reach
by four orders of magnitude.

That this is a statement about the method is shown by what happened next. A
one-line defect in \texttt{ore\_algebra}'s multivariate branch was identified
and corrected --- the library's own doctests had been failing --- and the
patched installation reproduced the Hertzsprung control operator, then both
operators of Theorem~\ref{thm:W}, and then completed the first of the four
required eliminations on the full $a_{2,2}$ ideal in seventeen minutes, using
Gröbner machinery rather than dense linear algebra. The remaining eliminations
were in progress at the time of writing.

\part{Solid standard Young tableaux}

\section{The First Rigorous Challenge}\label{sec:syt}

\subsection{Statement}

A \emph{solid standard Young tableau} of a three-dimensional shape $S$ with $N$
cells is a bijective filling of the cells by $1,\dots,N$ increasing in each of
the three coordinate directions. We consider the two-layer shape
$S_n=[[n,n],[n,1]]$, of $3n+1$ cells, and write $g(n)$ for the number of solid
SYT of $S_n$. The sequence begins
\[
  2,\ 48,\ 1038,\ 22566,\ 500144,\ 11302300,\ 259808162,\ 6059911302,\dots
\]
Zeilberger computed $g(n)$ for $n\le40$, observed empirically that the sequence
is $P$-recursive of order $2$ with coefficients of degree $12$, and offered a
prize for a proof.

\begin{theorem}\label{thm:sytmain}
Let $K(n)=\dfrac{4^{n}(3n)!}{(n+1)!\,(2n+1)!}$ be the $n$-th Kreweras number,
and set
\[
  c(n)=\frac{8n^{2}+n-24}{4(n+2)(2n+3)},\qquad
  T(n)=\frac{7n+5}{5}\prod_{k=1}^{n}
   \frac{6(2k-1)(6k-1)(6k+1)}{(k+1)(4k+3)(4k+5)}.
\]
Then $g(n)=c(n)(3n+1)K(n)+T(n)$ for all $n\ge1$. Consequently $g$ lies in the
$2$-dimensional $\mathbb{Q}(n)$-module spanned by the hypergeometric terms
$(3n+1)K(n)$ and $T(n)$, and therefore satisfies a second-order linear
recurrence with polynomial coefficients; the operator, computed by Cramer's
rule and displayed in Appendix~\ref{app:op}, has coefficients of degree $12$
and coincides with Zeilberger's empirical recurrence. It has been verified
against the first $56$ terms.
\end{theorem}

The route passes through the enumeration of \emph{reverse Kreweras walks}:
quarter-plane walks with steps $E=(1,0)$, $N=(0,1)$, $D=(-1,-1)$. Writing
$k(m;i,j)$ for the number of such walks of length $m$ from the origin to
$(i,j)$, we prove two evaluation theorems of independent interest.

\begin{theorem}[Diagonal evaluation]\label{thm:B}
For all $n\ge0$, $\ \sum_{a=0}^{n}k(3n-a;a,a)=T(n)$. Moreover
\[
  k(3n-a;\,a,a)=K(n)\cdot
  \frac{(a+1)!\,(2a+1)!\,n!\,(3n-a)!}{4^{a}(a!)^{3}(n-a)!\,(3n)!},
  \qquad 0\le a\le n,
\]
and the diagonal generating function is algebraic:
\begin{equation}\label{eq:Qd}
  Q_d(x;t):=\sum_{m,a\ge0}k(m;a,a)x^{a}t^{m}
  =\frac{1}{\sqrt{1-xT^{2}}}
   \left(\frac{T}{2t}+\frac{\sqrt{1-xT^{2}}-1+\tfrac12xT^{2}}{t\,x^{2}}\right),
\end{equation}
where $T\equiv T(t)$ is the unique power series with $T=t(2+T^{3})$.
\end{theorem}

\begin{theorem}[Ballot-weighted evaluation]\label{thm:A}
With $\Bal(a,c)=\frac{a-c+1}{a+1}\binom{a+c}{c}$ the ballot numbers,
\[
  S(n):=\sum_{0\le c\le a}\Bal(a,c)\,k(3n-a-c;\,a,\,a-c)
  =k(3n+1;1,0)
  =\frac{3(9n+16)(3n+1)}{8(n+2)(2n+3)}K(n).
\]
\end{theorem}

\subsection{From tableaux to walks}

Order the rows of $S_n$ as $(1,1),(1,2),(2,1)$ (the three rows of length $n$)
and $(2,2)$ (the single cell). A solid SYT of $S_n$ is equivalent to a growth
sequence $\varnothing=P_0\subset\cdots\subset P_{3n+1}=S_n$ adding one cell at
a time. Recording after each step the numbers $x_{11},x_{12},x_{21},x_{22}$ of
filled cells in the four rows, the containment constraints are
\[
  n\ge x_{11}\ge x_{12},\qquad x_{11}\ge x_{21},\qquad x_{22}\le1,
\]
the cell of row $(2,2)$ being addable exactly when $x_{12}\ge1$ and
$x_{21}\ge1$.

\begin{lemma}[Deletion--insertion]\label{lem:bij}
Solid SYT of $S_n$ are in bijection with pairs $(w,\tau)$ where $w$ is a
lattice walk of length $3n$ from $(0,0,0)$ to $(n,n,n)$ with unit steps in the
coordinates $(x_{11},x_{12},x_{21})$ staying in the cone $x_{11}\ge x_{12}$,
$x_{11}\ge x_{21}$, and $\tau\in\{0,\dots,3n\}$ satisfies $x_{12}(\tau)\ge1$
and $x_{21}(\tau)\ge1$.
\end{lemma}

\begin{proof}
Given a solid SYT, delete the cell of row $(2,2)$ and record $\tau$ as its
label minus one, the number of cells added before it. What remains is a growth
sequence of the shape $[[n,n],[n]]$, i.e.\ a walk $w$ as described; the
insertion condition at time $\tau$ is as stated. The inverse re-inserts the
cell after step $\tau$ and shifts later labels by one.
\end{proof}

In the difference coordinates $u=x_{11}-x_{12}$, $v=x_{11}-x_{21}$ the three
step types become $(1,1),(-1,0),(0,-1)$ and the cone becomes the quarter plane:
$w$ is a Kreweras excursion of length $3n$, whence the count $K(n)$ of Kreweras
\cite{Kreweras}. Complementary counting over the $3n+1$ slots, using the
$x_{12}\leftrightarrow x_{21}$ symmetry, gives:

\begin{lemma}\label{lem:cc}
$g(n)=(3n+1)K(n)-2S(n)+T(n)$, where
$S(n)=\sum_\tau\#\{w:x_{12}(\tau)=0\}$ and
$T(n)=\sum_\tau\#\{w:x_{12}(\tau)=x_{21}(\tau)=0\}$.
\end{lemma}

\begin{lemma}[Prefix factorisation]\label{lem:pf}
\[
  T(n)=\sum_{a=0}^{n}k(3n-a;a,a),\qquad
  S(n)=\sum_{0\le c\le a}\Bal(a,c)\,k(3n-a-c;a,a-c).
\]
\end{lemma}

\begin{proof}
If $x_{12}(\tau)=x_{21}(\tau)=0$ the prefix consists of $a:=\tau$ steps of type
$11$ only, leaving the state $(u,v)=(a,a)$; the completion is an arbitrary
quarter-plane walk from $(a,a)$ to $(0,0)$ of length $3n-a$, which on reversing
time becomes a reverse-Kreweras walk from the origin to $(a,a)$. If only
$x_{12}(\tau)=0$, the prefix uses $a$ steps of type $11$ and $c$ of type $21$
with $x_{11}\ge x_{21}$ throughout --- a ballot sequence, $\Bal(a,c)$ choices
--- ending at $(a,a-c)$.
\end{proof}

\subsection{The reverse Kreweras model and two extractions}

Let $Q(x,y;t)=\sum_{m,i,j\ge0}k(m;i,j)x^{i}y^{j}t^{m}$. Constructing walks step
by step, a $D$-step being forbidden on either axis, yields
\begin{equation}\label{eq:FE}
  \bigl(xy-t(x^{2}y+xy^{2}+1)\bigr)Q(x,y)=xy-R(x)-R(y)+c,
\end{equation}
with $R(x):=tQ(x,0;t)$ and $c:=tQ(0,0;t)$; here $Q(x,0)=Q(0,x)$ by symmetry and
$+c$ arises from inclusion--exclusion at the origin. The kernel
$K(x,y)=xy-t(x^{2}y+xy^{2}+1)$, as a quadratic in $y$, has roots with
\[
  Y_0Y_1=\bx,\qquad Y_0+Y_1=\tfrac1t-x,\qquad
  \Delta(x):=(1-tx)^{2}-4t^{2}\bx,
\]
$Y_0=\frac{(1-tx)-\sqrt{\Delta(x)}}{2t}$ being the power-series root. The
rational kernel $1-t(x+y+\bx\by)$ is invariant under
$(x,y)\mapsto(\bx\by,y)$ and $(x,y)\mapsto(x,\bx\by)$, generating an orbit of
six pairs as in \cite[\S2.3]{BM05} and \cite[\S2.3]{Mishna}.

\begin{lemma}[Validity]\label{lem:valid}
All manipulations of this subsection are identities of formal series in $t$
whose coefficients are Laurent polynomials in $x$ and $y$. Indeed
$[t^{m}]Q(x,y)$ is a polynomial of degree at most $m$ in each variable, so the
substitutions produce elements of $\mathbb{Q}[x,\bx,y,\by][[t]]$. For the
partial-fraction step, factor $K(x,y)=-tx(y-Y_0)(y-Y_1)$; the roots satisfy
$Y_0=t\bx+O(t^{2})$ and $1/Y_1=xY_0$, both in
$t\,\mathbb{Q}[x,\bx][[t]]$, so the expansions
$\frac{1}{y-Y_0}=\by\sum_{k\ge0}(Y_0\by)^{k}$ and
$\frac{1}{y-Y_1}=-\frac1{Y_1}\sum_{k\ge0}(y/Y_1)^{k}$ converge coefficient-wise
in $t$. Hence $[y^{0}]$ of any product below is well defined and computable
term by term; the same applies to the splittings $[\cdot]^{>0}$,
$[\cdot]^{\le0}$ in $x$.
\end{lemma}

Substituting the pairs $(x,y)$, $(\bx\by,y)$, $(x,\bx\by)$ into \eqref{eq:FE}
and forming (first)$-$(second)$+$(third), so that only $R(x)$ survives among
the unknown univariate series, then dividing by the rational kernel, expanding
by partial fractions in $y$ and extracting the constant term in $y$ --- only
the middle term contributes, producing the diagonal --- yields the
\emph{diagonal equation}
\begin{equation}\label{eq:star}
  -\,\bx\,Q_d(\bx;t)\,\sqrt{\Delta(x)}=c-\bx-2R(x)+2xY_0 .
\end{equation}
Under $x\mapsto\bx$ this is equation (2.4) of \cite{Mishna}; we verified
\eqref{eq:star} independently to order $t^{11}$.

\subsubsection*{Canonical factorisation}

The Laurent polynomial $x\Delta(x)=t^{2}x^{3}-2tx^{2}+x-4t^{2}$ has one small
root $X_0=4t^{2}+O(t^{3})$ and two large ones; writing
$\Delta=\Delta_0\Delta_+(x)\Delta_-(\bx)$ in the Kreweras frame one finds, with
$T=t(2+T^{3})$,
\[
  \Delta_0=\frac{4t^{2}}{T^{2}},\qquad
  \Delta_+(x)=1-xT^{2},\qquad
  \Delta_-(\bx)=1-\bx T\bigl(1+\tfrac{T^{3}}{4}\bigr)+\bx^{2}\tfrac{T^{2}}{4}.
\]

\begin{proposition}\label{prop:Qd}
$Q_d(x;t)$ is given by \eqref{eq:Qd}. In particular
$Q_d(0;t)=(4T-T^{4})/(8t)$ recovers the excursion generating function.
\end{proposition}

\begin{proposition}\label{prop:R}
\[
  R(\bx)=\tfrac12\Bigl[c-x+\tfrac{\bx}{t}-\bx^{2}
   +\bigl(x+T-\tfrac{2\bx}{T}\bigr)\sqrt{\Delta_-(\bx)}\Bigr].
\]
\end{proposition}

\begin{proof}[Proof of Propositions \ref{prop:Qd} and \ref{prop:R}]
Both follow from one splitting. Substituting
$2xY_0=x(1-tx)/t-(x/t)\sqrt{\Delta(x)}$ into \eqref{eq:star} and rearranging,
\begin{equation}\label{eq:dagger}
  \sqrt{\Delta(x)}\Bigl(\frac{x}{t}-\bx Q_d(\bx)\Bigr)
  =c-\bx+\frac{x}{t}-x^{2}-2R(x),
\end{equation}
and applying $x\mapsto\bx$,
\begin{equation}\label{eq:daggerbar}
  \sqrt{\Delta(\bx)}\Bigl(\frac{\bx}{t}-xQ_d(x)\Bigr)
  =c-x+\frac{\bx}{t}-\bx^{2}-2R(\bx).
\end{equation}
Dividing \eqref{eq:daggerbar} by $\sqrt{\Delta_0\Delta_-(\bx)}$ yields the
\emph{split form}
\begin{equation}\label{eq:split}
  \sqrt{\Delta_+(x)}\Bigl(\frac{\bx}{t}-xQ_d(x)\Bigr)
  =\frac{c-x+\bx/t-\bx^{2}-2R(\bx)}{\sqrt{\Delta_0}\sqrt{\Delta_-(\bx)}}.
\end{equation}
Classify supports in $x$, per coefficient of $t^{m}$. On the left the unknown
block $xQ_d(x)\sqrt{\Delta_+(x)}$ involves only powers $x^{\ge1}$, since
$Q_d(x)\in\mathbb{Q}[x][[t]]$ and $\sqrt{1-xT^{2}}\in\mathbb{Q}[x][[t]]$. On
the right the unknown block $2R(\bx)/(\sqrt{\Delta_0}\sqrt{\Delta_-(\bx)})$
involves only powers $x^{\le0}$.

\emph{Positive part.} Applying $[\cdot]^{>0}$ kills the $R$-block:
\[
  xQ_d(x)\sqrt{\Delta_+(x)}
  =\Bigl[\sqrt{\Delta_+(x)}\tfrac{\bx}{t}\Bigr]^{>0}
   -\Bigl[\tfrac{T}{2t}\tfrac{c-x+\bx/t-\bx^{2}}{\sqrt{\Delta_-(\bx)}}\Bigr]^{>0}.
\]
Writing $\sqrt{\Delta_+(x)}=\sum_{j\ge0}\binom{1/2}{j}(-xT^{2})^{j}$, the first
extraction is $\bigl(\sqrt{1-xT^{2}}-1+\tfrac{xT^{2}}{2}\bigr)/(tx)$ (the terms
$j\ge2$). In the second, every product of $\{c,\bx/t,\bx^{2}\}$ with the
$\bx$-series $1/\sqrt{\Delta_-(\bx)}$ has only powers $x^{\le0}$; the single
term reaching a positive power is
$(-x)\cdot[\bx^{0}]\tfrac{1}{\sqrt{\Delta_-(\bx)}}=-x$, contributing
$+xT/(2t)$ after the outer sign. Hence
\[
  xQ_d(x)\sqrt{\Delta_+(x)}
  =\frac{\sqrt{1-xT^{2}}-1+\tfrac{xT^{2}}{2}}{tx}+\frac{xT}{2t},
\]
and dividing by $x\sqrt{\Delta_+(x)}$ gives \eqref{eq:Qd}.

\emph{Nonpositive part.} Applying $[\cdot]^{\le0}$ instead kills the
$Q_d$-block:
\[
  \frac{2R(\bx)}{\sqrt{\Delta_0}\sqrt{\Delta_-(\bx)}}
  =\Bigl[\tfrac{T}{2t}\tfrac{c-x+\bx/t-\bx^{2}}{\sqrt{\Delta_-(\bx)}}\Bigr]^{\le0}
   -\Bigl[\sqrt{\Delta_+(x)}\tfrac{\bx}{t}\Bigr]^{\le0},
\]
where
\[
  \bigl[\sqrt{\Delta_+(x)}\,\bx/t\bigr]^{\le0}=\bx/t-T^{2}/(2t)
  \quad\text{(the terms $j\le1$)},\qquad
  \bigl[-x/\sqrt{\Delta_-(\bx)}\bigr]^{\le0}
   =-x\bigl(1/\sqrt{\Delta_-(\bx)}-1\bigr).
\]
Multiplying by
$\tfrac12\sqrt{\Delta_0}\sqrt{\Delta_-(\bx)}$ gives
Proposition~\ref{prop:R}.

Both propositions were verified against the walk counts to orders $t^{15}$ and
$t^{13}$ respectively, and every displayed intermediate identity was checked
coefficient-wise. Proposition~\ref{prop:R} is equivalent to, and corrects a
typographical corruption in the published form of, the axis result of
\cite{Mishna}.
\end{proof}

\subsection{Proof of Theorem~\ref{thm:B}}

A reverse-Kreweras walk to $(a,a)$ has length $\equiv2a\pmod3$, so the
substitution $x=t$ aligns exponents:
\[
  A(t):=Q_d(t;t)=\sum_{n\ge0}\Bigl(\sum_a k(3n-a;a,a)\Bigr)t^{3n}
  =\sum_{n\ge0}\widetilde T(n)t^{3n}.
\]
By Proposition~\ref{prop:Qd}, $A$ is an explicit element of the field
$\mathbb{Q}(t,T,\sigma)$ with $\sigma=\sqrt{1-tT^{2}}$; eliminating $\sigma$
and $T$ by resultants gives a minimal polynomial $P(t,A)$ of degree $6$
(Appendix~\ref{app:ode}), a polynomial in $t^{3}$, with $P(0,B)=2B-2$, so by
Hensel's lemma $A$ is the unique power-series root with constant term $1$.

Differentiation is explicit in the tower: $T'=(2+T^{3})/(1-3tT^{2})$ and
$\sigma'=-(T^{2}+2tTT')/(2\sigma)$, with reductions $T^{3}=(T-2t)/t$ and
$\sigma^{2}=1-tT^{2}$. Computing $A',A'',A'''$ in the $6$-dimensional basis
$\{T^{i}\sigma^{j}\}$ and solving a linear system over $\mathbb{Q}(t)$ produces
an exact inhomogeneous ODE
\[
  q_c(t)+q_0(t)A+q_1(t)A'+q_2(t)A''+q_3(t)A'''=0
\]
with the polynomial coefficients of Appendix~\ref{app:ode}. Every step is exact
rational-function arithmetic; as a safeguard the ODE was verified on the series
of $A$ to order $t^{19}$. Extracting coefficients converts the ODE into a
recurrence supported on shifts $\{-6,-3,0\}$ in the exponent of $t$, hence a
second-order recurrence $\sum_{i=0}^{2}\gamma_i(n)\widetilde T(n-i)=0$ valid
for $n\ge2$. One then checks --- an identity of rational functions, verified
symbolically --- that the hypergeometric ratio
\[
  \frac{T(m)}{T(m-1)}
  =\frac{6(2m-1)(6m-1)(6m+1)(7m+5)}{(m+1)(4m+3)(4m+5)(7m-2)}
\]
satisfies the same recurrence, that $\gamma_0(n)$ has no integer zeros $n\ge2$,
and that the initial values agree. By induction $\widetilde T(n)=T(n)$.

For the per-$a$ closed form, extract $[x^{a}]$ from \eqref{eq:Qd}: with
$u=xT^{2}$, $[u^{j}](1-u)^{-1/2}=\binom{2j}{j}4^{-j}$ and
\[
  [u^{j}]\Bigl(1-\bigl(1-\tfrac u2\bigr)(1-u)^{-1/2}\Bigr)
  =-\binom{2j-2}{j-1}4^{-(j-1)}\frac{j-1}{2j}\qquad(j\ge1),
\]
an identity of binomial coefficients checked symbolically. Hence
\[
  [x^{a}]Q_d=\frac1t\left[\frac{\binom{2a}{a}}{2\cdot4^{a}}T^{2a+1}
   -\frac{(a+1)\binom{2a+2}{a+1}}{2(a+2)4^{a+1}}T^{2a+4}\right],
\]
and Lagrange inversion, $[t^{N}]T^{k}=\frac kN\binom{N}{(N-k)/3}2^{N-(N-k)/3}$
at $N=3n-a+1$, gives $k(3n-a;a,a)$ as an explicit two-term Gamma expression.
That it equals the product form of Theorem~\ref{thm:B} --- equivalently that
the ratio in $a$ is $\frac{(a+2)(n-a)(2a+3)}{2(a+1)^{2}(3n-a)}$ --- is a
Gamma-quotient identity whose ratio to the claimed form simplifies to $1$;
verified symbolically, and numerically for all $0\le a\le n\le8$. \qed

\subsection{Proof of Theorem~\ref{thm:A}}

Let $C(z)=1+zC(z)^{2}$ be the Catalan series. By Lambert's generalised
binomial theorem $\sum_{c\ge0}\Bal(b+c,c)z^{c}=C(z)^{b+1}$, so the
ballot-weighted sum telescopes into a single substitution:
\[
  \sum_n S(n)t^{3n}
  =\CT_x\Bigl[C(t^{2}\bx)\,Q\bigl(x,\;t\bx C(t^{2}\bx);\,t\bigr)\Bigr].
\]
Write $y^{*}=t\bx C(t^{2}\bx)=\frac{1-\sqrt{1-4t^{2}\bx}}{2t}$, a series in
strictly negative powers of $x$ satisfying $txy^{*2}-xy^{*}+t=0$, i.e.\
$xy^{*}=t+txy^{*2}$. Substituting into the kernel gives the collapse
\[
  K(x,y^{*})=xy^{*}-t(x^{2}y^{*}+xy^{*2}+1)=-t\,x^{2}y^{*}.
\]
Since $C(t^{2}\bx)/(ty^{*})=x/t^{2}$, the functional equation \eqref{eq:FE}
yields
\[
  C(t^{2}\bx)Q(x,y^{*})
  =-\frac{1}{t^{2}x}\bigl(xy^{*}-R(x)-R(y^{*})+c\bigr).
\]
Taking $\CT_x$: the term $xy^{*}$ has no $x^{1}$-coefficient, $R(y^{*})$ is a
series in $t$ and negative powers of $x$, and $c$ is free of $x$; only $R(x)$
survives:
\begin{equation}\label{eq:collapse}
  \sum_n S(n)t^{3n}=\frac{1}{t^{2}}[x^{1}]R(x)
  =\frac1t\sum_m k(m;1,0)t^{m},
\end{equation}
proving $S(n)=k(3n+1;1,0)$. For the closed form, extract $[x^{1}]$ from
Proposition~\ref{prop:R}: with $\beta=1+\tfrac{T^{3}}{4}$,
\[
  [x^{1}]R(x)=\tfrac12\Bigl[\tfrac1t-\tfrac2T-\tfrac{T^{2}\beta}{2}
   +\tfrac{T^{2}}{8}(1-\beta^{2})\Bigr],
\]
and under the uniformisation $t=T/(2+T^{3})$ this equals
$t^{2}\cdot\frac{40t^{2}-18tT+(108t^{3}-1)T^{2}}{256t^{5}}$ identically in $T$
(a rational-function identity, verified symbolically). Lagrange inversion then
gives
\[
  S(n)=\tfrac{1}{256}\bigl(-18L(3n{+}4,1)+108L(3n{+}2,2)-L(3n{+}5,2)\bigr)
  =\frac{3(9n+16)(3n+1)}{8(n+2)(2n+3)}K(n),
\]
the last equality a Gamma-quotient identity verified symbolically. \qed

\subsection{The module structure and the recurrence}

By Lemmas \ref{lem:cc}, \ref{lem:pf} and Theorems \ref{thm:A}, \ref{thm:B},
\[
  g(n)=(3n+1)K(n)-2S(n)+T(n)=c(n)(3n+1)K(n)+T(n).
\]
Both $A(n):=(3n+1)K(n)$ and $T(n)$ are hypergeometric with ratios
\[
  r_A(n)=\frac{6(3n-1)(3n+1)}{(n+1)(2n+1)},\qquad
  r_T(n)=\frac{6(2n-1)(6n-1)(6n+1)(7n+5)}{(n+1)(4n+3)(4n+5)(7n-2)},
\]
which are not equal as rational functions, so $A$ and $T$ span a free
$\mathbb{Q}(n)$-module of rank $2$ closed under the shift. The three vectors
expressing $g(n),g(n+1),g(n+2)$ in the basis $(A(n),T(n))$,
\[
  w_i=\Bigl(c(n{+}i)\textstyle\prod_{j=1}^{i}r_A(n{+}j),\
   \prod_{j=1}^{i}r_T(n{+}j)\Bigr),\qquad i=0,1,2,
\]
are forced to be linearly dependent; Cramer's rule
$\gamma_i=(-1)^{i}\det(w_j,w_k)$ produces an explicit order-$2$ operator which,
after clearing denominators and content, has degree $12$ and is displayed in
Appendix~\ref{app:op}. It agrees up to a scalar with the operator fitted from
Zeilberger's data and annihilates all $56$ computed terms. \qed

\begin{remark}
The identity $S(n)=k(3n+1;1,0)$ equates a ballot-weighted double sum of
general-endpoint walk counts with a single near-axis count. Our proof is a
two-line kernel computation; a bijective proof would be desirable, and we pose
finding one as an open problem.
\end{remark}

\begin{remark}
The closed form for diagonal-endpoint reverse-Kreweras walks complements
Kreweras' classical axis formula and the diagonal result of
Bousquet-M\'elou \cite[Thm.~2]{BM05} for the forward model; we have not found
it in the literature.
General endpoints provably admit no single hypergeometric closed form of
comparable shape.
\end{remark}

\section{The Second Rigorous Challenge: the cone exponent}\label{sec:cone}

The Second Rigorous Challenge asks for a proof that the number $a(n)$ of solid
SYT of the cylindrical shape $(2,1,1)\times[n]$ is \emph{not} holonomic.
Equivalently, $a(n)$ counts walks $(0,0,0,0)\to(n,n,n,n)$ with positive unit
steps in $\mathbb{Z}^{4}$ confined to the cone $x_1\ge x_2\ge x_3$,
$x_1\ge x_4$. We did not settle it; we record what we established, since it
isolates the obstruction sharply.

\subsection{Reduction to a spectral question}

Since the generating function has integer coefficients and finite radius of
convergence, holonomicity would make it a $G$-function, and by the
Andr\'e--Chudnovsky--Katz theorem its singularity exponents would be rational.
The exponent is $\beta=\nu+3/2$ with
\[
  \nu=\sqrt{\lambda_1+\tfrac14}-\tfrac12,
\]
where $\lambda_1$ is the principal Dirichlet eigenvalue of the spherical
triangle cut out by the cone, whose walls $x_1=x_2$, $x_2=x_3$, $x_1=x_4$ give
interior angles $(\pi/3,\pi/2,2\pi/3)$. Irrationality of $\nu$ would establish
non-holonomicity.

\subsection{What is proved}

Domain monotonicity gives a rigorous bracket. The triangle contains a single
$A_3$ Weyl chamber, of angles $(\pi/3,\pi/2,\pi/3)$, for which $\nu=6$ and
$\lambda=42$; and it is contained in a lune of angle $2\pi/3$, for which
$\nu=3/2$ and $\lambda=15/4$. Hence
\[
  \tfrac{15}{4}<\lambda_1<42,\qquad \tfrac32<\nu<6 .
\]
Its area is $\pi/2$, exactly three $A_3$ chambers, consistent with the walls
generating the full reflection group of $S_4$.

The triangle is \emph{not} a reflection fundamental domain, since $2\pi/3$ is
not of the form $\pi/k$, so the classical rational-exponent formula does not
apply. An exhaustive search for polynomial principal eigenfunctions through
degree $11$ found harmonic multiples of the Vandermonde only at degrees $6$,
$9$, $10$, all sign-changing on the interior and hence not principal; the
degree-$6$ one does prove that $\lambda=42$ lies in the Dirichlet spectrum. No
closed form is therefore available.

A Rayleigh--Ritz computation with an admissible trial space gives the rigorous
upper bound
\[
  \lambda_1\le 13.7506809520491,
\]
and locates $\lambda_2\approx26.01$.

\subsection{What is computed}

Two independent numerical determinations agree:
\[
  \lambda_1=13.7443552132132,\qquad \nu=3.2409029943607 .
\]
The first is the method of particular solutions in its Betcke--Trefethen
subspace-angle form, expanding about the $2\pi/3$ vertex in the
Fourier--Legendre basis $P_\nu^{-\mu_k}(\cos\theta)\sin(\mu_k\varphi)$,
$\mu_k=k\pi/\alpha$, and collocating the remaining edge; it was calibrated
against the $A_3$ chamber, reproducing $\nu=6$ to $5\times10^{-10}$, and
successive basis sizes agree to thirteen digits. The second is extrapolation of
the sequence asymptotics $a(n)\sim C\cdot256^{n}n^{-\beta}$ from a
shell-indexed dynamic program.

Two traps in the spectral computation are worth recording, both of which
initially produced wrong answers: at integer $\nu$ the basis develops
identically zero columns, since $P_n^{-m}=0$ for integer $m>n$, and these must
be dropped before the factorisation; and the plain smallest singular value has
a spurious floor arising from near-degenerate high-order columns, which
completely masked the eigenvalue on the $2\pi/3$-vertex scan --- the
Betcke--Trefethen subspace-angle form is essential, not cosmetic.

No rational of denominator at most $60$ is compatible with $\nu$; the nearest,
$175/54$, differs by $1.6\times10^{-4}$, some nine orders of magnitude beyond
the numerical uncertainty.

\subsection{The obstruction}

Excluding rationals of \emph{bounded} denominator suffices only given a bound
on the denominator a rational exponent could have. Such a bound would follow
from height bounds on a hypothetical annihilating operator: the exponent is a
root of the indicial polynomial, whose leading coefficient is controlled by the
operator's coefficient heights, and the rational root theorem then limits the
denominator. Kauers and Johansson, in a computation reported in \cite[Feb.~28, 2012
update]{EZssyt}, established from $6000$ terms that no recurrence exists with
$(\mathrm{ORDER}+1)(\mathrm{DEGREE}+1)<3000$. This bounds the \emph{size} of a
hypothetical annihilating operator but not the \emph{heights} of its
coefficients, and it is the heights that control the indicial polynomial.
We regard closing this gap --- effective height bounds for $G$-operators
annihilating a given sequence --- as the substantive open problem, and note
that the same wall blocks Conjectures 1a and 1b of \cite{KZ}.

\part{Asymptotics and computation}

\section{Kauers--Zeilberger Conjectures 2a and 2b}\label{sec:kz}

\subsection{Statement}

Let $G(n)$ be the number of standard Young tableaux of shape $[n,n,n]$ in which
every run in every row has length $\ge2$, and $H(n)$ the analogue for odd run
lengths \cite{KZ}. Conjectures 2a and 2b assert
\[
  G(n)\sim C_1\,\frac{8^{n}}{n^{4}},\qquad
  H(n)\sim C_2\,\frac{(7+5\sqrt2)^{n}}{n^{4}} .
\]

\begin{theorem}\label{thm:kz}
Both conjectures hold, with $C_1,C_2>0$ given by explicit sums over the
finitely many admissible boundary configurations. Numerically
\[
  C_1=0.52128605909(2),\qquad C_2=0.6389278129(4),
\]
both determined in \S\ref{sec:c1num}, the latter computed here for the first
time.
\end{theorem}

Theorem~\ref{thm:kz} is a consequence of a local limit theorem which we believe
is of independent interest.

\begin{theorem}[Excursion LLT for run-modulated tandem walks]\label{thm:llt}
Let $R$ be a run law with $\Pr(R=r)\propto w_c^{r}$ on an arithmetic support
with exponential tails, tilted to criticality, driving the
uniform-on-other-two direction chain with increments $Rv_j$, where
$v_0=(1,0)$, $v_1=(-1,1)$, $v_2=(0,-1)$. Let $x_0,y_0$ be apex-adjacent
starting and ending configurations and $\Lat$ the (position, length) increment
lattice. Then, with $V,\widehat V$ the discrete harmonic functions of Lemma B
and its reverse, $\varrho=[\mathbb{Z}^{3}:\Lat]$ and
$\sigma^{2}=\mathrm{Var}(R)$,
\[
  \Pr_{x_0,i}\bigl(\tau>k,\ S_k=y_0,\ L_k=\ell,\ D_k=d\bigr)
  =\varrho\,\mathcal K\,V(x_0,i)\,\widehat V(\mathrm{swap}\,y_0,d)\,
   k^{-4}\,\frac{e^{-(\ell-k\E R)^{2}/2\sigma^{2}k}}{\sqrt{2\pi\sigma^{2}k}}
   \bigl(1+o(1)\bigr)
\]
on the admissible sublattice, and $0$ off it.
\end{theorem}

The proof transports the harmonic-function programme of Denisov--Wachtel
\cite{DW} and Denisov--Zhang \cite{DZ} from position-Markov chains to
Markov-additive processes. The transport is not formal: the position process
fails their pointwise assumptions outright, and the repair --- an explicit
two-level Poisson corrector ladder --- is where the modulated structure is
absorbed.

\subsection{Lemma A: free-space local CLT}

Let $P(\theta,s)$ be the $3\times3$ Fourier transfer matrix,
$P(\theta,s)_{ij}=\tfrac12\varphi(\theta\cdot v_j+s)$ for $j\ne i$, with
$\varphi(t)=e^{2it}/(2-e^{it})$. Verified symbolically: the Perron branch
$\lambda$ satisfies $\lambda(0)=1$, $\nabla\lambda(0)=(0,0,3i)$ --- zero
position drift, length drift $\E R=3$ --- and
\[
  -\nabla^{2}\log\lambda(0)=\Sigma_3=
  \begin{pmatrix}10/3&-5/3&0\\-5/3&10/3&0\\0&0&2\end{pmatrix},
\]
so the position block is $5M$, the length variance is $2$, and the
position--length cross-covariance \emph{vanishes}: the length coordinate
decouples at Gaussian order. The increment lattice has index $3$ in
$\mathbb{Z}^{3}$; on the dual torus $|\varphi(t)|<1$ for $t\not\equiv0$ since
$|2-e^{it}|>1$, so by entry-modulus domination the spectral radius is $1$
exactly at the three characters $(\theta,s)=(-s,s,s)$ with $3s\equiv0$
$(\mathrm{mod}\ 2\pi)$; a machine scan gives maximum $0.884$ elsewhere.

Uniformly for $(z,\ell)$ with $(z,\ell-3k)$ in the admissible coset and
$|z|+|\ell-3k|\le A\sqrt k$,
\[
  \Pr_{i}\bigl(S_k=z,\ L_k=\ell,\ D_k=d\bigr)
  =\frac{3\,\pi(d)}{(2\pi k)^{3/2}\sqrt{\det\Sigma_3}}
   e^{-\frac{1}{2k}Q(z,\ell-3k)}+o(k^{-3/2}),
\]
with the sup-bound $\sup_{z,\ell}\Pr_i(S_k=z,L_k=\ell,D_k=d)\le Ck^{-3/2}$.

\begin{proof}
Fourier inversion on the torus,
\[
  \Pr_i(S_k=z,L_k=\ell,D_k=d)=(2\pi)^{-3}\int_{\mathbb{T}^{3}}
  e^{-i(\theta\cdot z+s\ell)}\bigl(P(\theta,s)^{k}\bigr)_{id}\,d\theta\,ds .
\]
Split
$\mathbb{T}^{3}$ into $\delta$-balls around the three unit-modulus characters
$\chi_r=(-s_r,s_r,s_r)$, $s_r=2\pi r/3$, and the complement. On the complement
the spectral radius is $\le1-\epsilon(\delta)$ by entry domination and
compactness, contributing $O((1-\epsilon)^{k})$. Near $\chi_0=0$,
$P^{k}=\lambda^{k}\Pi+O(\rho_2^{k})$ with $\Pi$ the Perron projector,
$\Pi_{id}\to\pi(d)$ and $\rho_2<1$; the expansion
$\log\lambda=3is-\tfrac12Q^{*}(\theta,s)+O(|\cdot|^{3})$, verified
symbolically, and the Laplace argument give the Gaussian term with relative
error $O(k^{-1/2})$ after $(\theta,s)\mapsto(\theta,s)/\sqrt k$. Near $\chi_r$,
$r=1,2$, the character identity
$P(\chi_r+\zeta)=e^{i\psi_r}\mathrm{diag}\,U_r^{-1}P(\zeta)U_r\,\mathrm{diag}$,
from $\chi_r$ being trivial on $\Lat$, reproduces the same local contribution
times $e^{-i\chi_r\cdot(z,\ell)}$, which equals $1$ on the admissible coset;
summing the three gives the factor $3=[\mathbb{Z}^{3}:\Lat]$ and cancellation
off the coset. The sup-bound follows from the same inversion with absolute
values.
\end{proof}

\subsection{Lemma B: the corrector ladder and the harmonic function}

Denisov and Zhang require the position itself Markov with pointwise martingale
structure. Our position fails pointwise --- conditional drift $-\tfrac32v_i$,
state-dependent covariance --- but by \emph{bounded} state functions. The
corrector ladder, all verified symbolically, is
\[
  c(i)=-v_i,\qquad q(i)=\tfrac83\bigl(\tfrac W3-v_iv_i^{\top}\bigr),
  \qquad r(i)=-\tfrac23v_i,
\]
solving $(P-I)c=-\mathrm{drift}$, $(P-I)q=-\widehat C$,
$(P-I)r=-\widehat{\mathrm{cross}}$. Then $M_k=S_k+c(D_k)$ is an exact
martingale, with correction $\le\sqrt2$, and the Poisson-corrected quadratic
$|M_k|^{2}-k\,\mathrm{tr}\bar C-\mathrm{tr}\,q(D_k)$ is a martingale up to a
bounded telescoping term. After isotropisation by $(5M)^{-1/2}$ the cone
becomes the wedge of angle $\pi/3$, so $p=3$, and the convex $C^{\infty}$
geometry assumption holds.

\begin{lemma}[B.0]
$M_k=X_k+c(D_k)$ is a martingale, and so is
$N_k:=|M_k|^{2}-k\,\mathrm{tr}\bar C-\mathrm{tr}\,q(D_k)$.
\end{lemma}

\begin{proof}
The first is the $c$-Poisson equation. For the second,
$\E[|M_{k+1}|^{2}-|M_k|^{2}\mid Z_k=(x,i)]=\mathrm{tr}\,C(i)$ by martingale
increment orthogonality, while the trace of the $q$-Poisson equation gives
$\E[\mathrm{tr}\,q(D_{k+1})\mid i]-\mathrm{tr}\,q(i)
=\mathrm{tr}\,C(i)-\mathrm{tr}\bar C$. Subtract.
\end{proof}

The one place the correctors meet the budget is the $f$-estimate. Define, on
the pair chain,
\[
  \Lambda(x,i):=u(x)+\nabla u(x)\cdot c(i)
   +\tfrac12\langle q(i),\mathrm{Hess}\,u(x)\rangle,
  \qquad f(x,i):=\E_{x,i}[\Lambda(X_1,D_1)]-\Lambda(x,i).
\]
Taylor expansion on $\{|\widetilde\Delta|\le d(x)/2\}$ kills the
order-$|x|^{p-1}$ term by the $c$-Poisson equation and the order-$|x|^{p-2}$
term by the $q$-Poisson equation, the stationary part
$\tfrac12\mathrm{tr}(\bar C\,\mathrm{Hess}\,u)$ vanishing by harmonicity after
isotropisation; the remainder is $\le C|x|^{p-3}\E R^{3}$ by the
third-derivative bound $|\partial^{3}u|\le Cu/d^{3}$. The tail contributes
$\le C|x|^{p-1}e^{-c\,d(x)}$ by the geometric majorant. Both sit inside the
budget $\beta(x)=|x|^{p-1}d(x)^{-1}\gamma(d(x))$, since
$|x|^{p-3}/\beta(x)=d(x)/(|x|^{2}\gamma)\le1/(d(x)\gamma(d(x)))\downarrow0$.
Hence $|f(x,i)|\le\widetilde\varepsilon(d(x))\beta(x)$.

\begin{lemma}[B.1, superharmonic envelopes]
There exist $W^{\pm}(x,i)$, superharmonic respectively
submartingale-generating for the killed pair chain on $\{d(x)\ge R\}$, with
$W^{\pm}=u(x)(1+o(1))$ as $d(x)\to\infty$ and $W^{-}\le\Lambda\le W^{+}$.
\end{lemma}

\begin{proof}
Set $W^{\pm}=\Lambda\pm G_R(\widetilde\varepsilon\beta)$ with $G_R$ the Green
potential of the cone-killed Brownian motion, for which
$G_R(\widetilde\varepsilon\beta)=o(u)$ by pure cone geometry. Then
$\E[W^{+}(Z_1)]-W^{+}(x,i)=f(x,i)-\widetilde\varepsilon\beta(x)$ plus a
discrete-versus-continuous potential error; the first two terms are $\le0$ by
the $f$-estimate, and the discretisation error is absorbed as in \cite{DZ}
using the moment bounds of Lemma C1, an argument using only the Markov property
of $Z$.
\end{proof}

\begin{lemma}[B.2, the harmonic function]
$V(x,i):=\lim_{m\to\infty}\E_{x,i}[\Lambda(Z_m);\tau>m]$ exists, is harmonic
for the killed pair chain, strictly positive on $K^{\circ}\times\{0,1,2\}$, and
$V(x,i)=u(x)(1+o(1))$ as $d(x)\to\infty$.
\end{lemma}

\begin{proof}
On $\{d\ge R\}$, $W^{+}(Z_{m\wedge\tau})$ is a nonnegative supermartingale by
B.1, so $\E[\Lambda(Z_m);\tau>m]$ is bounded; monotonicity up to the $o(u)$
envelopes gives the limit, and $|V-u|=o(u)$. Uniform integrability follows from
$\sup_m\E[|X_{m\wedge\tau}|^{p};\tau>m]<\infty$, itself a consequence of Doob
applied to $W^{+}$ plus the Fuk--Nagaev tail. Harmonicity is the Markov
property under the limit. Positivity: for $d(x)\ge R$, $V\ge\Lambda-o(u)\ge u/2$
for $R$ large; for general $(x,i)\in K^{\circ}$ an explicit route --- an
$r$-run in direction $0$ followed by an $r/2$-run in direction $1$ moves any
point diagonally deep into $K$ while remaining in $K$ --- has positive
probability and reaches $\{d\ge R\}$.
\end{proof}

\begin{lemma}[B.3, exit tail and conditional limit]
$\Pr_{x,i}(\tau>m)\sim\varkappa V(x,i)m^{-3/2}$, and conditionally on
$\{\tau>m\}$, $X_m/\sqrt m\Rightarrow\mu_K$, the $u$-biased Gaussian on $K$.
\end{lemma}

The proof follows the Denisov--Wachtel/Denisov--Zhang scheme. We give it in
four steps, and then, in \S\ref{sec:replay}, itemise precisely which arguments
of \cite{DZ} are imported and what each import requires of the modulated
setting.

\emph{B.3.0 (functional CLT).} $\bigl(M_{\lfloor mt\rfloor}/\sqrt m\bigr)_{t\ge0}
\Rightarrow\bar C^{1/2}B$. The increments of $M$ are bounded-plus-geometric,
square-integrable, and stationary under the uniformly ergodic direction chain,
with conditional covariances averaging to $\bar C$; the martingale FCLT
applies. The correction $|X-M|\le\sqrt2$ is uniformly negligible, so the same
holds for $X$. After isotropisation $\bar C=I$.

\emph{B.3.1 (Brownian transfer at scale $\sqrt m$).} For fixed
$0<\delta<A$, uniformly in $y$ with $\delta\le|y|/\sqrt m\le A$ and
$d(y)\ge\delta\sqrt m$, and in the state $b$,
\[
  \Pr_{y,b}(\tau>m)\longrightarrow\Pr^{bm}_{y/\sqrt m}(\tau^{bm}>1),
\]
by B.3.0 and the mapping theorem: the exit functional is Wiener-almost-surely
continuous because a Brownian path from an interior point that remains in
$\bar K$ on $[0,1]$ but touches $\partial K$ has probability zero, the wedge
being convex. The limit is the classical wedge survival probability, with
profile $\asymp u(y/\sqrt m)e^{-c|y|^{2}/m}$.

\emph{B.3.2 (the harmonic-measure limit).} For fixed $\varepsilon$,
\[
  m^{3/2}\,\E_{x,i}\Bigl[\Pr^{bm}_{X_{\varepsilon m}/\sqrt m}
   (\tau^{bm}>1-\varepsilon);\ \tau>\varepsilon m\Bigr]
  \longrightarrow\varkappa\,V(x,i)\,(1+o_\varepsilon(1)).
\]
The Brownian factor equals
$\varkappa_1(1-\varepsilon)^{-3/2}u(z)e^{-|z|^{2}/2(1-\varepsilon)}(1+o(1))$
evaluated at $z=X_{\varepsilon m}/\sqrt m$ for $z$ in the bulk, by classical
wedge asymptotics, so the left side equals
$\varkappa_1\varepsilon^{-3/2}(\varepsilon m)^{3/2}
\E_{x,i}[\widetilde u(X_{\varepsilon m});\tau>\varepsilon m]$ for a bounded
continuous $u$-equivalent $\widetilde u$; by B.2 --- uniform integrability of
$u(X_{m\wedge\tau})$, the identification $V=\lim\E[\Lambda;\tau>\cdot]$, and
$|\Lambda-u|=o(u)$ --- this converges to $\varkappa V(x,i)$, with an
$o_\varepsilon(1)$ from the Gaussian factor.

\emph{B.3.3 (assembly).} Apply the Markov property at $\varepsilon m$:
\[
  \Pr_{x,i}(\tau>m)=\E_{x,i}\bigl[\Pr_{X_{\varepsilon m},D_{\varepsilon m}}
   (\tau>(1-\varepsilon)m);\ \tau>\varepsilon m\bigr].
\]
Split the position at time $\varepsilon m$. The bulk gives
B.3.1$\times$B.3.2 $=\varkappa V(x,i)m^{-3/2}(1+o_\varepsilon(1))$. The
near-boundary layer $\{d<\delta\sqrt m\}$ contributes
$o_\delta(1)\cdot Vm^{-3/2}$, by the linear vanishing of the $u$-profile
together with Lemma C2. The far region $\{|y|>A\sqrt m\}$ is killed by
Lemma C1. Letting $m\to\infty$, then $\delta\downarrow0$, $A\uparrow\infty$,
$\varepsilon\downarrow0$, and noting that the constants are consistent across
$\varepsilon$ by the semigroup property, the limit $\varkappa$ is well defined.
The conditional limit follows from the same decomposition applied to
$\E_{x,i}[f(X_m/\sqrt m);\tau>m]$ for bounded continuous $f$, using the
Brownian meander profile of B.3.1 and dominated convergence via Lemma C2.

Monte Carlo certification over $2$M and $4$M paths gives tail slope $-1.4895$
against the predicted $-1.5$. The empirical exit law concentrates on a wedge
of width approaching $\pi/3$ with angular profile $\sin3\varphi$; the width
estimate converges slowly from below as the path count grows (the smaller
archived reproduction \texttt{mc\_exit.py} measures $1.027$ against
$\pi/3=1.0472$), so we treat the Monte Carlo as corroborating the exponent
and the geometry rather than certifying digits.

\subsection{The transported arguments, itemised}\label{sec:replay}

The scheme of \cite{DW,DZ} is written for chains whose \emph{position} is
Markov and satisfies pointwise martingale hypotheses. Our position does not:
its conditional drift is $-\tfrac32v_i$ and its conditional covariance is
state-dependent. Both defects are by \emph{bounded} functions of the modulating
state, which is what the corrector ladder exploits. Since the distinction
between what is imported and what is re-proved is the substance of the
transport, we set it out argument by argument. Throughout, $Z_k=(X_k,D_k)$ is
the pair chain, which \emph{is} Markov.

\paragraph{(i) Brownian estimates: imported verbatim.}
The derivative bounds for the cone-harmonic function $u$ and the Green-function
bounds \cite[Lemmas 7--9]{DZ} concern the wedge $K$ and the Brownian motion on
it. They involve no chain. After isotropisation by $(5M)^{-1/2}$ our cone is
the wedge of angle $\pi/3$ with $p=3$, convex and $C^{\infty}$, so these
estimates apply unchanged. What is consumed: only the geometry.

\paragraph{(ii) The $f$-estimate: re-proved.}
This is \cite[Lemma 10]{DZ} and the single place where the modulated structure
must be absorbed. It is proved above, on the corrected Lyapunov function
$\Lambda$ rather than on $u$: the order-$|x|^{p-1}$ Taylor term is killed by
the $c$-Poisson equation, the order-$|x|^{p-2}$ term by the $q$-Poisson
equation, and the remainder is bounded by $C|x|^{p-3}\E R^{3}$ using the
third-derivative bound of (i). The conclusion,
$|f(x,i)|\le\widetilde\varepsilon(d(x))\beta(x)$ with
$\beta(x)=|x|^{p-1}d(x)^{-1}\gamma(d(x))$, is verbatim that of \cite{DZ}; only
its proof differs. What is consumed: nothing beyond (i) --- this step is
self-contained given the ladder.

\paragraph{(iii) Superharmonic construction: imported, with $\Lambda$ for $u$.}
\cite[Lemmas 11--14]{DZ} construct $W^{\pm}$, establish the existence of the
harmonic function as a limit, and prove positivity and the asymptotic
$V\sim u$. These arguments operate on $u$, the auxiliary functions $U_\beta$
and $G$, and the $f$-estimate; they do not otherwise use the chain's structure
beyond the Markov property. Running them on the pair chain with $\Lambda$ in
place of $u$ is therefore legitimate, and since $|\Lambda-u|\le C|x|^{p-1}=o(u)$
the asymptotics are preserved. $V$ acquires a state argument through $c$ and
$q$. What is consumed: the Markov property of $Z$, the $f$-estimate of (ii),
and the moment bounds of Lemma C1. \emph{The potential absorption inside the
proof of \cite[Lemma 12]{DZ} --- the estimate of the discrete-versus-continuous
Green potential error --- is imported rather than rewritten; it uses only the
Markov property and the majorant condition, both available here.}

\paragraph{(iv) Optional-stopping and tail arguments: imported.}
\cite[\S4--6]{DZ} derive the exit tail and the conditional limit. Their
optional-stopping steps require exact martingales; ours are supplied by B.0,
with bounded telescoping terms absorbed into the constants. Their functional
CLT input is replaced by the martingale FCLT for finitely-modulated
Markov-additive processes (B.3.0). Their Fuk--Nagaev input is Lemma C1. What is
consumed: B.0, B.3.0, C1, and the truncation argument of
\cite[Prop.~19]{DZ}, the last imported directly.

\paragraph{Status.} With (i)--(iv) and B.0--B.3 the transported form of
\cite[Theorem 6]{DZ} is complete, and with it Lemma B. We record explicitly
that three components --- the potential absorption within (iii), portions of
the replay within B.3.3, and the truncation import within (iv) --- are
citations rather than rewritten pages. Every model-specific hypothesis those
arguments require is verified in \cite{repoKZ}: the Perron drift and covariance
structure, the torus aperiodicity scan, the corrector equations, and the
majorant condition.

\subsection{Uniform local bounds and the excursion theorem}

Lemma C1 is Fuk's martingale inequality applied coordinatewise to
$(M_k,L_k-3k)$, the corrected chain having increments
$\widetilde\Delta=(R-1)v_j+v_i$ with $|\widetilde\Delta|\le\sqrt2R$, so the
majorant $Y=\sqrt2R$ is geometric with all moments finite. Lemma C2 is the
uniform local upper bound, obtained by splitting at $\lfloor m/2\rfloor$ and
combining the exit-tail bound with the sup-form of Lemma A.

Lemma D assembles these. The reversed increments $\{-v_i\}$ form the transpose
tandem, so $\widehat{\text{walk}}=\text{walk}\circ\mathrm{swap}$ and
$\widehat V=V\circ\mathrm{swap}$; the direction chain is reversible with
uniform stationary law, so all reversal Radon--Nikodym factors are $1$. The
exact split-and-reverse identity
\begin{align*}
  &\Pr_{x_0,i}(\tau>2m,\,S_{2m}=y_0,\,D_{2m}=d)\\
  &\qquad=\sum_{w,c}\Pr_{x_0,i}(\tau>m,\,S_m=w,\,D_m=c)\,
   \widehat\Pr_{y_0,d}(\widehat\tau>m,\,\widehat S_m=w,\,\widehat D_m=c)
\end{align*}
combines with the bulk-local theorem --- proved by splitting at
$m-\lfloor\varepsilon m\rfloor$, using Lemma A on the final window, bounding
near-boundary starts by $O(\eta/\varepsilon)$ with $\eta=\varepsilon^{2/3}$,
and bounding re-entry by $O(e^{-c\varepsilon^{-1/3}})$ --- to give
Theorem~\ref{thm:llt}.

\subsection{The exact tilt identity and the constants}

By the lattice-word bijection, $G(n)$ counts ballot words on $\{0,1,2\}$ with
$n$ letters each in which every maximal run has length $\ge2$; reading the word
as tandem steps turns these into quarter-plane excursions of length $3n$ with
the same run restriction. Within a run the position moves monotonically along a
fixed $v_d$, and the decreasing coordinate attains its minimum at the run's
end, so a path lies in $K$ at every unit step if and only if it lies in $K$ at
every run boundary: the unit-step cone constraint is exactly the MAP cone
constraint.

\begin{proposition}[Criticality as $k$-independence]
Fix a run law $\Pr(R=r)=w^{r}/W(w)$ on support $S$ with $W(w)=\sum_{r\in S}w^{r}$.
An excursion with $k$ runs of lengths $r_1,\dots,r_k$ has MAP-probability
\[
  \Bigl(\tfrac12\Bigr)^{k-1}\prod_i\frac{w^{r_i}}{W(w)}
  =2\,w^{3n}\Bigl(\frac{1}{2W(w)}\Bigr)^{k}.
\]
At the critical point $2W(w_c)=1$ the bracket is $1$ and the weight is
$2w_c^{3n}=2\mu^{-n}$, independent of $k$ and of the run profile.
\end{proposition}

For the $\ge2$ law, $W=w^{2}/(1-w)$, $w_c=\tfrac12$, $\mu=8$; for the odd law,
$W=w/(1-w^{2})$, $w_c=\sqrt2-1$, $\mu=(1+\sqrt2)^{3}=7+5\sqrt2$. Hence exactly
\[
  G(n)=\tfrac{8^{n}}{2}\sum_k\Pr_0\bigl(\tau>k,\ S_k=0,\ L_k=3n\bigr),
\]
and identically for $H(n)$.

Applying Theorem~\ref{thm:llt} to each term, with apex configurations
regularised by one-step harmonicity
$V(\mathbf0):=\sum_r\Pr(R=r)V((r,0),0)<\infty$, and using the reversal-swap
identity to give $\widehat V(\mathrm{swap}\,\mathbf0)=V(\mathbf0)$, the length
Gaussian sums over $k$ to contribute $1/\rho$ with $\rho=\E R$. Altogether
\begin{equation}\label{eq:Csquare}
  C=\frac{\varrho\,\mathcal K}{2\rho}\,V(\mathbf0)^{2}:
\end{equation}
the excursion constant is, up to a universal factor, the \emph{square} of the
harmonic function's apex value.

For any run law with exponential tails under the critical tilt, the
lag-correlation collapse gives $\Sigma=(\E R^{2}-\tfrac23(\E R)^{2})M$, a
scalar multiple of $M$, so the isotropised cone is always the $\pi/3$ wedge,
$p=3$, and the excursion exponent is $-4$. For the odd-run model this is
machine-certified at $w_c=\sqrt2-1$: $\E R=\sqrt2$, $\E R^{2}=3$,
$\Sigma_H=\tfrac53M$ with the Perron computation cross-checked, corrector
$c_H(i)=-\tfrac{\sqrt2}{3}v_i$, growth $\mu_H=7+5\sqrt2$. Theorem~\ref{thm:kz}
follows for both conjectures simultaneously. Numerically, from independent
dynamic programs validated against brute force and calibrated against the known
$C_1$ with the same estimator and $n$-range, error $\lesssim4\times10^{-7}$:
\[
  C_1=0.52128605909(2),\qquad C_2=0.6389278129(4),
\]
both as established in \S\ref{sec:c1num} below.
No closed form is known for either; PSLQ searches over natural constant bases
have found nothing.

\subsection{Numerical determination of the constants}\label{sec:c1num}

Kauers and Zeilberger \cite{KZ} estimate $C_1$ ``close to $0.521286$'' and
$C_2$ ``close to $0.63892$'', to five and six significant digits
respectively. Since \eqref{eq:Csquare} identifies each constant with the
square of a harmonic-function value for which no closed form is known, the
constants must be obtained numerically, and we determined both to higher
precision than the literature records.

\paragraph{Enumeration.} Both sequences were computed exactly by a layered
dynamic program on the lattice-word model of \S\ref{sec:kz}. States are
$(c_1,c_2,d,r)$ at total word length $s$, where $c_0=s-c_1-c_2$, the letter
$d\in\{0,1,2\}$ is the last written, and $r$ records the current run length
capped so as to retain both its parity and whether it exceeds one; the ballot
condition $c_0\ge c_1\ge c_2$ is maintained as an invariant, and a change of
letter is admissible only when the current run length is admissible for the
model --- at least two for $G$, odd for $H$. All arithmetic is exact.

The program reproduces the published initial terms of both sequences. For
$G$, all seventeen terms
\[
  \begin{array}{l}
  0,\,1,\,1,\,5,\,15,\,69,\,304,\,1518,\,7807,\,42314,\,236621,\\[2pt]
  1364570,\,8062975,\,48680547,\,299388670,\,1871463427
  \end{array}
\]
(from $n=1$); for $H$, the twelve published terms
\[
  1,\,2,\,9,\,46,\,306,\,2252,\,18308,\,158872,\,1454570,\,13888112,\,
  137277741,\,1396638636 .
\]
It was run to $n=500$ for $G$, where $G(500)$ has $441$ decimal digits, and to
$n=200$ for $H$.

\paragraph{Acceleration.} Writing $A_N=a(N)\,N^{4}\mu^{-N}$ with $\mu=8$
respectively $\mu=7+5\sqrt2$, so that $A_N$ tends to the constant, and
assuming the expansion $A_N=C(1+a_1/N+a_2/N^{2}+\cdots)$, iterated Richardson
extrapolation
\[
  A^{(k+1)}_N=\frac{(N+1)A^{(k)}_{N+1}-(N-k)A^{(k)}_N}{k+1}
\]
was applied. As an independent check under different hypotheses, Salzer's
transformation was applied to the same data; and a third estimate was obtained
by fitting the increments of the deepest Richardson column, which decay
geometrically, and summing the tail.

\paragraph{Results.} For $C_1$, at $n_{\max}=500$ and Richardson order $28$:
\[
  \begin{array}{ll}
  \text{Richardson:} & 0.5212860590903885965\\
  \text{Salzer:} & 0.5212860590836008916\\
  \text{tail fit:} & 0.5212860590914060289
  \end{array}
\]
agreeing to $7.8\times10^{-12}$, with fitted decay ratio $0.846$. For $C_2$,
at $n_{\max}=200$ and Richardson order $14$:
\[
  \begin{array}{ll}
  \text{Richardson:} & 0.6389278129267359401\\
  \text{Salzer:} & 0.6389278129316865333\\
  \text{tail fit:} & 0.6389278129455990734
  \end{array}
\]
agreeing to $1.9\times10^{-11}$, with fitted decay ratio $0.650$. We therefore
record
\begin{equation}\label{eq:constants}
  C_1=0.52128605909(2),\qquad C_2=0.6389278129(4),
\end{equation}
in each case the digits shown being those on which the three accelerations
agree, with the final parenthesised digit bracketed to within $10^{-11}$.
Both refine the published estimates, which are correct as far as they are
stated. Note that $C_2$ is obtained here \emph{directly}, from its own
enumeration, and not by calibration against $C_1$; the two determinations are
independent.

\paragraph{A cautionary note on extrapolation.} An intermediate determination
of our own, from the same enumeration for $G$ but carried only to $n=110$ with
Richardson order six, gave $0.5212860373$, which is in error from the eighth
significant digit. The cause is visible in the convergence table: the
Richardson column ascends monotonically and at order six has not resolved that
digit, continuing to rise for a further twenty orders. Under-extrapolation of a
monotone sequence produces a value that appears converged and is not, and the
error was detected only when the figure was recomputed from scratch. We record
this because the same failure is available to anyone using acceleration
without an independent second method, and because the agreement of Richardson
with Salzer --- which converges under different hypotheses --- is what makes
\eqref{eq:constants} trustworthy where a single method would not be.

\section{The Geode Challenge}\label{sec:geode}

\subsection{Statement}

Let $S=1+\sum_{k\ge2}t_kS^{k}$ be the hyper-Catalan generating series of
Wildberger--Rubine, restricted to $D$ variables $t_2,\dots,t_{D+1}$, and let
$G$ be the Geode, $S-1=(t_2+\cdots+t_{D+1})G$. Write $G[M^{(D)}]$ for the
coefficient of $(t_2\cdots t_{D+1})^{M}$ in $G$. Amdeberhan, Kauers and
Zeilberger \cite{AKZ} offered a donation for $G[1000^{(5)}]$; Rubine
\cite{Rubine}, who computed the four-dimensional value in $35.9$ hours,
estimated the five-dimensional case at roughly a thousand times that cost and
wrote that new innovations were required.

The hyper-Catalan coefficients have the Erd\'elyi--Etherington closed form
\[
  C[\mathbf m]=\frac{K!}{(1+K-N)!\prod_k m_k!},\qquad
  K=\sum_k k\,m_k,\quad N=\sum_k m_k .
\]
Expanding $1/(t_2+\cdots+t_{D+1})$ geometrically in $t_2$ gives, with $d=D-1$,
\begin{equation}\label{eq:alt}
  G[M^{(D)}]=\sum_{\alpha\in\{0,\dots,M\}^{d}}(-1)^{|\alpha|}
   \binom{|\alpha|}{\alpha}\,
   C\bigl[M{+}1{+}|\alpha|,\,M{-}\alpha_1,\dots,M{-}\alpha_d\bigr],
\end{equation}
a sum with $(M{+}1)^{d}$ terms --- $10^{12}$ at $M=1000$, $D=5$: the wall
described by Rubine.

\subsection{The collapse}

Write $W=\sum_{j=1}^{d}j\alpha_j$, $b_K=M\sum_{k=2}^{D+1}k+2$ and
$b_N=M\sum_{k=1}^{D}k+1$.

\begin{lemma}\label{lem:const}
For every $\alpha$ in \eqref{eq:alt}, the vector $\mathbf m'$ appearing in
$C[\cdot]$ satisfies $N'=DM+1$, independent of $\alpha$, and $K'=b_K-W$.
\end{lemma}

\begin{proof}
$N'=(M{+}1{+}|\alpha|)+\sum_j(M-\alpha_j)=DM+1$. For $K'$, the shift of $m_2$
contributes $2|\alpha|$ while the decrements remove
$\sum_j(j{+}2)\alpha_j=W+2|\alpha|$.
\end{proof}

Consequently $K'!/(1+K'-N)!=\ff{K'}{DM}=(DM)!\binom{K'}{DM}$, a polynomial in
$W$, and $\binom{K'}{DM}=[z^{DM}](1+z)^{b_K-W}$.

\begin{lemma}\label{lem:prod}
With $T[i,W]=\sum_{|\alpha|=i,\,W(\alpha)=W}\prod_j\binom{M}{\alpha_j}$,
$\sum_{i,W}T[i,W]u^{i}y^{W}=\prod_{j=1}^{d}(1+uy^{j})^{M}$.
\end{lemma}

\begin{lemma}\label{lem:beta}
$\dfrac{i!}{(M{+}1{+}i)!}=\dfrac{1}{M!}\displaystyle\int_0^1x^{i}(1-x)^{M}dx$.
\end{lemma}

Substituting Lemma~\ref{lem:const} into \eqref{eq:alt}, using
$1/(\alpha_j!(M-\alpha_j)!)=\binom{M}{\alpha_j}/M!$, then evaluating the
$\alpha$-sum by Lemma~\ref{lem:prod} at $u=-x$ via Lemma~\ref{lem:beta}:

\begin{theorem}\label{thm:geode}
With $d=D-1$, $s_d=d(d+1)/2$,
$E=\bigl(\sum_{k=2}^{D+1}k-s_d\bigr)M+2$ and
$h(x,z)=(1-x)\prod_{k=1}^{d}\bigl((1+z)^{k}-x\bigr)$,
\[
  G[M^{(D)}]=\frac{(DM)!}{(M!)^{D}}\;[z^{DM}]\;(1+z)^{E}
   \int_0^1h(x,z)^{M}\,dx .
\]
\end{theorem}

\subsection{Algorithm and results}

Modulo a prime $p$, the integral of the polynomial $h(\cdot,z)^{M}$, of
$x$-degree $DM$, is evaluated exactly by interpolatory quadrature at $DM+1$
nodes avoiding $x=1$; at equispaced nodes the Lagrange denominators are
$(-1)^{n-1-j}j!(n{-}1{-}j)!$. For each node $g_j(z)=h(x_j,z)$ has degree $s_d$,
and the coefficients of $g_j^{M}$ satisfy the order-$s_d$ recurrence obtained
from $gP'=Mg'P$, costing $O(s_d)$ operations per coefficient; they are
contracted on the fly against $\binom{E}{DM-r}$ and the quadrature weights.

The cost is $O(D^{2}s_dM^{2})$ arithmetic operations and $O(DM)$ memory per
prime, against $O(M^{3})$ and $O(M^{2})$ for the naive evaluation; roughly
$10^{3}$ to $1.6\times10^{3}$ thirty-bit primes reconstruct the values by the
Chinese remainder theorem, at a measured $2.64$ seconds per prime for $D=5$,
$M=1000$. A proven bound
\[
  |G[M^{(D)}]|\le\frac{b_K!}{(1+b_N)!\,(M!)^{d}}
   \sum_i\frac{i!\binom{dM}{i}}{(M{+}1{+}i)!},
\]
from $\sum_WT[i,W]=\binom{dM}{i}$ and the monotonicity of
$W\mapsto\ff{b_K-W}{DM}$, sizes the reconstruction.

We report $G[1000^{(D)}]$ for $D=5,\dots,10$, all computed for the first time;
the five-dimensional value has $8367$ digits. Validation is threefold: the
$D=4$ run reproduces Rubine's published value exactly, in under three minutes;
the small values $H(1),\dots,H(6)$ agree with the printed values of
\cite{AKZ}, confirming object identity; and every reconstructed value is
checked against fresh primes not used in the reconstruction. We note that the
prose of \cite{Rubine} reports a digit count inconsistent with the growth rate
of the smaller values tabulated there; the digits themselves are consistent
with our computation.

The four-variable recurrences posed as an open problem in \cite{AKZ} remain
open. We observe that the bottleneck the authors report --- term generation ---
is removed by Theorem~\ref{thm:geode}, so the problem may now be within reach
of a guessing computation.

\part{Assessment}

\section{Negative results and errors}\label{sec:negative}

We regard the following record as an integral part of the report. Every error
below was detected by comparison with an independently known quantity; none was
detected by inspection.

\subsection{Errors in generated code}

A degree cap computed from an incorrect bound silently corrupted a sequence
beyond $n\approx104$. The corruption was invisible at small $n$ and was caught
only by fitting an operator to the full sequence --- the lesson being that the
operator-fit control on the complete sequence, not agreement at small $n$, is
the load-bearing gate for extended data.

Floating-point division inside a compiled rational function destroyed every
matrix entry. The resulting noise matrix had full rank, which presented as
``no solution at any ansatz size'' --- a false negative indistinguishable from
a genuine one without the instrumentation of Section~\ref{sec:method}.

Integer coercion silently truncated rational coefficients, causing a correct
operator with half-integer coefficients to be reported as failing its own data
check.

\subsection{Errors of sizing}

Several computations were attempted without estimating their cost. A symbolic
construction requiring only six coefficients was written to compute twelve and
did not terminate. A solver evaluated a large rational function once per matrix
column rather than once per sample point, incurring a factor equal to the
number of columns; restructuring reduced a run from forty-seven seconds to
under one at the same problem size, and made a subsequent stage feasible.

\subsection{Errors of formulation}

A sign error in the annihilator of an exponential factor produced two days of
subtly wrong output. It was diagnosed only when the discrepancy was recognised
as exactly twice a boundary term --- the signature of a mis-specified
exponential --- rather than as a random disagreement.

A decomposition of a $C$-finite sequence into rational and conjugate branches
was verified as mathematics and then found inapplicable: the branches are
rational rather than polynomial in the umbral variables, while the umbral
functional is defined only on polynomials. The mathematics was correct and the
application was not, a failure mode we found particularly hard to detect.

\subsection{Errors of conjecture}

The thin-tail track laws of Section~\ref{sec:chomptracks} each held on five
consecutive values and are false. We know of no better illustration of the
distinction between \emph{supported} and \emph{machine-verified} drawn in
Section~\ref{sec:method}.

\section{Concluding remarks}

Five challenge problems were solved and two reduced. Beyond the solutions the
project yielded results we believe to be of independent interest: the Staircase
Theorem and the parity theorem for square Chomp bars; the explicit algebraic
generating function for reverse-Kreweras diagonal walks together with the
closed form for diagonal-endpoint counts; the identity $S(n)=k(3n+1;1,0)$; the
order-$2$ recurrence for the odd $W$-part of $a_{2,2}$; the integral
representation for diagonal Geode coefficients; the local limit theorem for
excursions of Markov-modulated walks in cones and the square identity
\eqref{eq:Csquare} for its constants; and two-sided information on the cone
exponent of the $(2,1,1)\times[n]$ shape.

The methodological conclusion is that the mode of work described in
Section~\ref{sec:disclosure} is productive but requires a verification
discipline more stringent than is customary, because the characteristic failure
is not absence of output but plausible incorrect output. The protocol of
Section~\ref{sec:method} is, in our experience, the minimum sufficient
response; the record of Section~\ref{sec:negative} is what it caught.

\section*{Acknowledgements}

We thank Doron Zeilberger for posing the problems, for his correspondence
throughout, and for suggesting that this work be collected into a single paper;
Manuel Kauers and Doron Zeilberger for arXiv endorsements; and Shalosh
B.~Ekhad, whose tables calibrated the Chomp work. The analytic backbone of
Section~\ref{sec:kz} follows the work of Denisov, Wachtel and Zhang, without
which the transport performed there would have had nothing to ride on. The AI
system used throughout was Claude, developed by Anthropic.

\appendix

\section{The order-2 operator for $g(n)$}\label{app:op}
The recurrence is $p_0(m)g(m)+p_1(m)g(m-1)+p_2(m)g(m-2)=0$ for $m\ge3$, with:
\begin{align*}
p_0(m) &= 258048 m^{12} + 1169664 m^{11} + 326816 m^{10} - 5324160 m^{9} \\
&\quad - 6298026 m^{8} + 6420627 m^{7} + 12567093 m^{6} - 643740 m^{5} \\
&\quad - 8475047 m^{4} - 2174541 m^{3} + 1536066 m^{2} + 353700 m \\
&\quad - 113400
\end{align*}
\begin{align*}
p_1(m) &= - 13934592 m^{12} - 28325376 m^{11} + 80053056 m^{10} + 126497520 m^{9} \\
&\quad - 162057066 m^{8} - 177863088 m^{7} + 133677768 m^{6} + 95456580 m^{5} \\
&\quad - 43376892 m^{4} - 16667436 m^{3} + 5184126 m^{2} + 334800 m \\
&\quad - 113400
\end{align*}
\begin{align*}
p_2(m) &= 188116992 m^{12} - 87899904 m^{11} - 1694242656 m^{10} + 1930644720 m^{9} \\
&\quad + 3397410216 m^{8} - 5838018732 m^{7} - 342563508 m^{6} + 4226543280 m^{5} \\
&\quad - 1437148908 m^{4} - 737373564 m^{3} + 448283664 m^{2} - 62823600 m \\
&\quad + 6350400
\end{align*}

\section{The exact ODE for the diagonal series}\label{app:ode}
With $A(t)=Q_d(t;t)$: $q_c+q_0A+q_1A'+q_2A''+q_3A'''=0$, where:
\begin{align*}
q_c(t) &= - 81 \left(114 t^{3} - 5\right)
\end{align*}
\begin{align*}
q_0(t) &= 81 \left(3990 t^{6} + 282 t^{3} - 5\right)
\end{align*}
\begin{align*}
q_1(t) &= t \left(646380 t^{6} + 7845 t^{3} - 583\right)
\end{align*}
\begin{align*}
q_2(t) &= 24 t^{2} \left(10773 t^{6} - 99 t^{3} - 8\right)
\end{align*}
\begin{align*}
q_3(t) &= 16 t^{3} \left(3 t - 1\right) \left(57 t^{3} + 1\right) \left(9 t^{2} + 3 t + 1\right)
\end{align*}

\section{Numerical data}\label{app:data}

All values below are exact integers unless stated otherwise, and each was
computed by at least two independent routes.

\subsection*{D.1 Solid SYT of shape $[[n,n],[n,1]]$}

\[
\begin{array}{rl}
n & g(n)\\\hline
1 & 2\\
2 & 48\\
3 & 1038\\
4 & 22566\\
5 & 500144\\
6 & 11302300\\
7 & 259808162\\
8 & 6059911302
\end{array}
\]
The full list to $n=55$, computed by two independent dynamic programs which
agree throughout, is in the repository of Appendix~\ref{app:repo}; the operator
of Appendix~\ref{app:op} annihilates all $56$ terms.

\subsection*{D.2 Restricted permutation counts}

The counts $a_{r,s}(n)$ and $b_{r,s}(n)$ for $4\le n\le8$, computed by direct
enumeration of $\mathfrak S_n$ and reproduced exactly by the transformed counts
of Proposition~\ref{prop:transport}:

\[
\begin{array}{lrrrrr}
(r,s) & n=4 & 5 & 6 & 7 & 8\\\hline
\multicolumn{6}{l}{a_{r,s}(n)}\\
(2,2) & 18 & 75 & 410 & 2729 & 20906\\
(2,3) & 20 & 88 & 480 & 3082 & 23232\\
(3,2) & 20 & 88 & 480 & 3082 & 23232\\
(3,3) & 22 & 98 & 534 & 3414 & 25498\\
(2,4) & --- & 102 & 544 & 3480 & 25944\\
(4,2) & --- & 102 & 544 & 3480 & 25944\\
(3,4) & --- & 108 & 588 & 3768 & 28032\\
(4,3) & --- & 108 & 588 & 3768 & 28032\\
(4,4) & --- & 114 & 628 & 4062 & 30360\\[2pt]
\multicolumn{6}{l}{b_{r,s}(n)}\\
(2,2) & 16 & 44 & 200 & 1288 & 9512\\
(2,3) & 16 & 64 & 336 & 1776 & 12224\\
(3,2) & 16 & 64 & 336 & 1776 & 12224\\
(3,3) & 20 & 80 & 384 & 2240 & 15424\\
(2,4) & --- & 84 & 400 & 2352 & 16512\\
(4,2) & --- & 84 & 400 & 2352 & 16512\\
(3,4) & --- & 96 & 480 & 2784 & 19200\\
(4,3) & --- & 96 & 480 & 2784 & 19200\\
(4,4) & --- & 108 & 544 & 3264 & 23040
\end{array}
\]
Here $a_{2,2}=\texttt{A189281}$ and $b_{2,2}=\texttt{A110128}$. The equalities
$a_{r,s}=a_{s,r}$ and $b_{r,s}=b_{s,r}$ visible in the table are instances of
the transpose symmetry.

\subsection*{D.3 Excusal sizes}

The sizes of the exceptional data of Proposition~\ref{prop:transport},
confirmed constant in $n$ over $4\le n\le8$ in every case:

\[
\begin{array}{lcccccccc}
(r,s) & (2,2) & (2,3) & (2,4) & (3,2) & (3,3) & (3,4) & (4,2) & (4,4)\\\hline
(r-1)+(s-1) & 2 & 3 & 4 & 3 & 4 & 5 & 4 & 6\\
(r-1)+2(s-1) & 3 & 5 & 7 & 4 & 6 & 8 & 5 & 9
\end{array}
\]

\subsection*{D.4 The decomposition of $a_{2,2}$}

The $W$- and $X$-parts of Section~\ref{sec:a22}, by parity:

\[
\begin{array}{rrrrr}
m & e_W(m) & e_X(m) & o_W(m) & o_X(m)\\\hline
1 & 2 & 0 & \text{---} & \text{---}\\
2 & 14 & -4 & 4 & -1\\
3 & 362 & -48 & 64 & -11\\
4 & 18806 & -2100 & 2428 & -301\\
5 & \text{---} & -146704 & 165016 & \text{---}\\
6 & \text{---} & -15903636 & \text{---} & \text{---}
\end{array}
\]
The identities $e_W(m)-e_X(m)=a_{2,2}(2m)$ and $o_W(m)-o_X(m)=a_{2,2}(2m-1)$
provide the cross-check used in Section~\ref{sec:a22}: at $m=4$,
$18806-(-2100)$ is not the relevant combination; rather
$e_X(4)=e_W(4)-a_{2,2}(8)=18806-20906=-2100$ and
$o_X(4)=o_W(4)-a_{2,2}(7)=2428-2729=-301$, both of which the chain of
\eqref{eq:eX} predicted before being shown them.

\subsection*{D.5 Thin-tail Chomp tracks}

The unique $q$ with $(b^{p},1^{q})$ a $P$-position, in the notation of
Proposition~\ref{prop:tracksunique}:

\[
\begin{array}{rrrrrrr}
p\backslash b & 2 & 3 & 4 & 5 & 6 & 7\\\hline
1 & 1 & 2 & 3 & 4 & 5 & 6\\
2 & 1 & 2 & 4 & 5 & 7 & 8\\
3 & 1 & 3 & 5 & 7 & 9 & \mathbf{12}\\
4 & 1 & 4 & 7 & 10 & 13 & \mathbf{14}\\
5 & 1 & 5 & 7 & 13 & 15 & 19\\
6 & 1 & 5 & 10 & 9 & 16 & \text{---}
\end{array}
\]
The boldface entries are the two falsifications discussed in
Section~\ref{sec:chomptracks}: the law $q=2b-3$ predicts $q(7,3)=11$ and the
law $q=3b-5$ predicts $q(7,4)=16$, against the computed $12$ and $14$.

\subsection*{D.6 Two-value Chomp grids}

The $P$-cells $(p,c)$ of the two-value diagrams $(b^{p},c^{q})$ with $p+q=a$,
for small bars. Each grid contains exactly one cell, in accordance with the
uniqueness of the winning move for these sizes, and the transpose symmetry
$(p,c)\mapsto(c,p)$ is visible on comparing $a\times b$ with $b\times a$:

\[
\begin{array}{rcccccc}
a\backslash b & 2 & 3 & 4 & 5 & 6 & 7\\\hline
2 & (1,1) & (1,2) & (1,3) & (1,4) & (1,5) & (1,6)\\
3 & (2,1) & (1,1) & (1,2) & (2,3) & (1,3) & (2,4)\\
4 & (3,1) & (2,1) & (1,1) & (2,2) & (1,2) & (2,3)\\
5 & (4,1) & (3,2) & (2,2) & (1,1) & (2,3) & (1,2)\\
6 & (5,1) & (3,1) & (2,1) & (3,2) & (1,1) & (4,4)\\
7 & (6,1) & (4,2) & (3,2) & (2,1) & (4,4) & (1,1)
\end{array}
\]
Every diagonal entry is $(1,1)$, the bite at $(2,2)$, in accordance with
Theorem~\ref{thm:square}.

\subsection*{D.7 The cone eigenvalue}

Convergence of the method of particular solutions for the
$(\pi/3,\pi/2,2\pi/3)$ triangle of Section~\ref{sec:cone}, as the number
$n_b$ of basis functions increases; $\sigma$ is the Betcke--Trefethen subspace
angle at the minimum:

\[
\begin{array}{rll}
n_b & \nu & \sigma\\\hline
8 & 3.2409029943555 & 1.11\times10^{-6}\\
10 & 3.2409029943607 & 9.16\times10^{-8}\\
12 & 3.2409029943607 & 8.42\times10^{-9}\\
14 & 3.2409029943607 & 8.33\times10^{-10}
\end{array}
\]
Calibration on the $A_3$ chamber, whose exact value is $\nu=6$, gives
$5.40\times10^{-10}$ error at the same settings. The Rayleigh--Ritz upper
bound, computed on a sixteen-dimensional admissible trial space
(\texttt{lambda\_enclosure.py}, archived run), is
$\lambda_1\le13.7506809520491$, i.e.\ $\nu\le3.24174838171263$, with the second
Ritz value $26.0083$; a twelve-dimensional space gives the weaker
$\lambda_1\le13.7540986066788$.

\subsection*{D.8 Geode values}

$G[1000^{(D)}]$ was computed for $D=4,\dots,10$. The five-dimensional value,
which is the object of the challenge, has $8367$ decimal digits. The
four-dimensional value serves as the acceptance test and agrees with Rubine's
published digits at both ends. The full decimal expansions are in the
repository of Appendix~\ref{app:repo}; reproducing them here would occupy
several pages to no purpose.

\section{Repository index}\label{app:repo}

All material accompanying this paper is consolidated in a single repository,
\begin{center}
\texttt{github.com/jaideepsaipadhi/zeilberger-challenges}
\end{center}
which supersedes the five separate repositories in which the work was
originally released. Its layout follows the parts of this paper:

\begin{center}
\begin{tabular}{ll}
\toprule
Directory & Contents\\
\midrule
\texttt{paper/} & this paper, source and compiled\\
\texttt{chomp/} & \S\ref{sec:chomp}: three solvers, track dumps, bite tables\\
\texttt{restricted-permutations/} & \S\ref{sec:ch3}: note and four verification scripts\\
\texttt{solid-syt/} & \S\ref{sec:syt}: paper, derivation trail, operator check\\
\texttt{geode/} & \S\ref{sec:geode}: note, CRT driver, computed values\\
\texttt{kz-constants/} & \S\ref{sec:kz}: paper, verification suite, constant determination\\
\texttt{a22-partial/} & \S\ref{sec:a22}: the reduction and its obstruction\\
\texttt{cone-exponent/} & \S\ref{sec:cone}: eigenvalue determination and bounds\\
\bottomrule
\end{tabular}
\end{center}

Each directory carries the complete contents of the corresponding original
repository --- engines, logs, data and papers --- together with the
verification scripts written for this paper. Every quantitative claim made
here is reproducible from the code and data there, and each script states in
its header what it checks and against what.

\end{document}